\documentclass[11pt]{amsart}
\usepackage[utf8]{inputenc}

\usepackage{graphicx}

\usepackage{amsmath, amssymb, amsthm, amsbsy,  tikz}
\usepackage{amsfonts}
\usepackage{fullpage}
\usepackage{hyperref}
\usepackage{geometry}
\usepackage{color}
\usepackage{comment}

\newcommand{\ZZ}{{\mathbb{Z}}}

\newcommand\A{{\mathcal A}}

\renewcommand{\S}{\mathfrak{S}}

\newcommand\zzn{{\ZZ_3^n \times \ZZ}}

\DeclareMathOperator\Sym{\rm Sym}
\DeclareMathOperator\Ker{\rm Ker}
\DeclareMathOperator\Stab{\rm Stab}

\DeclareMathOperator\tC{\widetilde{C}}
\DeclareMathOperator\tA{\widetilde{A}}
\DeclareMathOperator\tB{\widetilde{B}}
\DeclareMathOperator\tD{\widetilde{D}}

\newtheorem{theorem}{Theorem}[section]
\newtheorem{proposition}[theorem]{Proposition}
\newtheorem{lemma}[theorem]{Lemma}
\newtheorem{corollary}[theorem]{Corollary}
\newtheorem{example}[theorem]{Example}

\newtheorem{question}[theorem]{Question}
\newtheorem{defn}[theorem]{Definition}

\newtheorem{remark}[theorem]{Remark}
\newtheorem{problem}[theorem]{Problem}

\newtheorem{observation}[theorem]{Observation}

\newcommand{\todo}[1]{\vspace{5 mm}\par \noindent
	\marginpar{\textsc{ToDo}} \framebox{\begin{minipage}[c]{0.95
				\textwidth}
			#1 \end{minipage}}\vspace{5 mm}\par}

\title{Combinatorial flip actions \\ and Gelfand pairs for affine Weyl groups}

\author{Ron M. Adin}
\address{Department of Mathematics, Bar-Ilan University, 
	Ramat-Gan 52900, Israel}
\email{radin@math.biu.ac.il}

\author{P\'al Heged\"us}
\address{R\'enyi Alfr\'ed Institute of Mathematics, Re\'altanoda utca 13-15, H-1053, Budapest, Hungary}
\email{hegedus.pal@renyi.hu}

\author{Yuval Roichman}
\address{Department of Mathematics, Bar-Ilan University, 
	Ramat-Gan 52900, Israel}
\email{yuvalr@math.biu.ac.il}

\dedicatory{%We dedicate this paper 
	To the memory of Jan Saxl, our dear friend, mentor and source of inspiration}

\date{November 29, 2021}%\today}

\thanks{PH was partially supported by Hungarian National Research, Development and Innovation Office (NKFIH), Grant No.\ K115799, and by a visiting grant from Bar-Ilan University. The project leading to this application has received funding from the European Research Council (ERC) under the European Union's Horizon 2020 research and innovation programme, Grant agreement No.\ 741420.
	RMA and YR were partially supported by the Israel Science Foundation, Grant No.\ 1970/18.}

\begin{document}
	
	\maketitle
	
	\begin{abstract}
		Several combinatorial actions of the affine Weyl group of type $\tC_{n}$ on triangulations, trees, words and permutations are compared.  
		Addressing a question of David Vogan, 
		%whether the resulting $\tC_{n}$-modules are multiplicity-free
		we show that, modulo a natural involution, 
		%inversion, 
		these permutation representations are multiplicity-free. 
		%Proofs are obtained via a %generalized 
		The proof uses a general construction of Gelfand subgroups  
		in the affine Weyl groups of types $\tC_{n}$ and $\tB_n$. 
		%This construction is then generalized to affine Weyl groups of type $\tB_n$.
		%other affine Weyl groups of classical type {\bf [to be slightly modified ??]}.
		%multiplicity-free modules. 
	\end{abstract}
	
	\tableofcontents
	
	%\todo{In Proposition~\ref{prop:Cn-words} the combinatorial definition of the action-equivariant involution is missing. See the lacuna in the figure.}
	
	\section{Introduction}
	
	%\todo{PH: \textbf{I suggest: Weyl group $G$ of type $\tC_n$ or we may use letter $W_n$ throughout the whole paper. As $\tC$ is a type, not a good name for a group itself.}}
	
	%Recall that the 
	The affine Weyl group of type $\tC_{n}$ is generated by
	\[
	S=\{s_0,s_1,\ldots, s_{n}\}
	\]
	subject to the Coxeter relations
	\begin{align*}
	s_i^2=1 \qquad & \forall i,\\
	(s_i s_j)^2=1 \qquad &  \text{for } |j-i| > 1,\\
	(s_i s_{i+1})^3=1 \qquad &  \text{for } 1\le i \le n-2,\\
	(s_i s_{i+1})^4=1 \qquad &  \text{for } i=0 \text{ and } n-1.
	\end{align*}
	
	In this paper we present several actions of $\tC_{n}$ on combinatorial objects ---
	triangulations, trees, words, and permutations ---  
	and address a question of David Vogan inquiring whether the resulting $\tC_{n}$-modules are multiplicity-free.
	We build a uniform framework which contains all the above-mentioned actions as special cases.
	It turns out that the answer to the original question is negative. 
	However, there is a $\tC_{n}$-equivariant pairing of the objects (which has a nice combinatorial interpretation in each case) 
	such that the induced action on the set of pairs 
	is indeed multiplicity-free.
	%It is shown that these 
	%actions (modulo a natural involution) are representatives of a wide family
	%modules, modulo a natural involution, 
	%%%Rewritten by PH for comment#1%%%
	%indeed belongs to a large family of multiplicity-free $\tC_{n}$-modules.
	
	\bigskip
	
	%We begin with 
	Here is a typical example.
	For a set $X$, let $\Sym(X)$ 
	be the group of 
	all bijections of $X$ onto itself.
	In particular, let $\S_n:=\Sym([n])$ 
	be the symmetric group of permutations on the set $[n]:=\{1,\dots,n\}$.
	%Denote a permutation $\pi\in \S_n$ by the sequence
	%$[\pi(1),\dots,\pi(n)]$. 
	% and transpositions by $(i,j)$.
	%Let
	%$\sigma_i:=(i,i+1)$ - a simple reflection, $S:=\{\sigma_i\mid\ 1\le i< n\}$ -
	%the Coxeter generating set, and $w_0:=[n,n-1,\dots,1]$ - the
	%longest permutation (with respect to the Coxeter length).
	%{\em Intervals} in the cyclic group $\ZZ_n$  are subsets of the form
	%%$[i,i+k]:=
	%$\{i, i+1, \dots, i+k\}$, where addition is modulo $n$.
	A permutation in $\S_n$ is an {\em arc permutation} if every prefix forms an interval in the cyclic group $\ZZ_n$; see Subsection~\ref{sec:arc} for more details.
	%(where the letter $n$ is identified with zero). 
	Denote by $\A_n$ the set of arc permutations in $\S_n$. 
	
	\medskip
	
	%We now describe a natural action of the 
	The affine Weyl group $\tC_{n}$ acts on the set of arc permutations
	$\A_{n+2}$ as follows.
	%Denote the adjacent transposition $(i,i+1)$  by $\sigma_i$.
	
	\begin{defn}%\label{def:c-action}
		For every $0\le i\le n$, let $\sigma_{i+1} \in \S_{n+2}$ be the transposition $(i+1,i+2)$.
		Define a group homomorphism $\rho: \tC_{n} \to \Sym(\A_{n+2})$ 
		%as the multiplicative extension of the following map:
		by
		\[
		\rho(s_i)(\pi) :=
		\begin{cases}
		\pi \sigma_{i+1},&\text{if }\pi \sigma_{i+1}\in \A_{n+2};\\
		\pi, &\text{otherwise}
		\end{cases}
		\qquad (\forall \pi \in \A_{n+2},\, 0 \le i \le n).
		\]
	\end{defn}

	%\begin{proposition}\label{t.action}
	%The maps $\rho: \A_{n+2}\to \A_{n+2} $
	%%$\rho_i,\ 0\le i\le n$, 
	%%when extended multiplicatively,
	%determines a well-defined $\widetilde C_{n}$-action on
	%the set of arc permutations $\A_{n+2}$.
	%\end{proposition}

	It was shown in~\cite{TFT2} that this %the above map 
	determines a well-defined transitive 
	$\tC_{n}$-action on $\A_{n+2}$.
	The resulting Schreier graph and its diameter were studied in~\cite{TFT2, ER, Roy}.%%%Fixed by PH for comment#3%%%
	
	\bigskip
	
	%\begin{corollary}
	%The above $\widetilde C_{n-2}$-action on $\A_n$ is transitive.
	%\end{corollary}
	
	%Let $J:=S\setminus \{s_{0}\}=\{s_1,\ldots, s_{n-3}, s_{n-2}\}$.
	%Recall that the maximal parabolic subgroup $W_J$
	%is isomorphic to the hyperoctahedral group $B_{n-2}$.
	
	%\begin{proposition}
	%\begin{itemize}
	%\item[(i)] The maps $\rho_i,\ 0< i\le n-2$, when extended
	%multiplicatively,
	%determine a well-defined $B_{n-2}$-action on $\A_n$.
	
	%\item[(ii)] The orbits of this action are $\{\pi\in \A_n\mid\pi(1)=k\}$, for $1\le k\le n$.
	
	%\item[(iii)] The $B_{n-2}$-action on each of these orbits is
	%multiplicity-free.
	%\end{itemize}
	%\end{proposition}

	The following question was posed by David Vogan (personal communication, 2010).
	
	\begin{question}~\cite[Question 1]{ER}\label{qu:original} Is the $\tC_{n}$-module determined by this action on $\A_{n+2}$ multiplicity-free?
	\end{question}
	
	Before we continue we must say a few words about multiplicity-freeness. A finite dimensional module $V$ of a finite group $G$ is \emph{multiplicity-free} if every simple module of $G$ occurs in $V$ with multiplicity at most $1$. Suppose that $G$ is not finite, but $\Ker(V)=\{g\in G\mid g(v)=v,\,\forall v\in V \}$, the kernel of $V$ is of finite index in $G$. As $V$ is naturally a $G/\Ker(V)$-module, we may still call $V$ multiplicity-free if it is multiplicity-free as a $G/\Ker(V)$-module in the original sense. In our situation $\tC_n$ is not finite, indeed. Nevertheless, it acts on the finite set $\A_{n+2}$, so each stabilizer $\Stab_{\tC_n}(\pi)$ is of finite index. As $\Ker(V)=\cap_{\pi}\Stab_{\tC_n}(\pi)$ is the intersection of finitely many subgroups of finite index, hence $|\tC_n:\Ker(V)|<\infty$. So Question~\ref{qu:original} and all other similar claims made in the paper make sense.%%%Inserted by PH for comment#11.%%%
	
	In this paper we construct a family of multiplicity-free $\tC_{n}$-modules, and deduce an ``almost affirmative'' answer to the above question:
	while the
	$\tC_{n}$-module determined by the above action on $\A_{n+2}$ is not multiplicity-free, the following holds.
	
	\medskip
	
	%  prove the following theorem.
	
	Let $\iota:\S_{n+2}\mapsto \S_{n+2}$  be the involution defined by
	%$\pi^\iota$ be the permutation 
	\begin{equation}\label{eq:def_L}
	\pi^{\iota}:= w_0 \pi \qquad (\forall \pi\in \S_{n+2}),  %= [n+3-\pi(1),\dots,n+3-\pi(1)],
	\end{equation}
	where $w_0=[n+1,n,\ldots,1,n+2]$ is the longest element in  maximal parabolic subgroup $\S_{n+1}:=\{\pi\in \S_{n+2}\mid \pi(n+2)=n+2\}$.
	In other words, for $\pi=[\pi(1),\dots,\pi(n+2)]\in \S_{n+2}$ 
	%, let defined by
	%\begin{equation}\label{eq:def_L}
	\[
	\pi^\iota(i) :=
	\begin{cases}
	n+2-\pi(i) &\text{if }1\leq \pi(i) \leq n+1; \\
	n+2 &\text{if }\pi(i) = n + 2.
	\end{cases}
	\]
	%\[
	%\pi^\iota(i):=n+2-\pi(i) \qquad (1\le i\le n+2),    
	%\]
	%\end{equation}
	%where zero is identified with $n+2$.
	%%%Rewritten by PH for comment#4%%%
	
	%denote
	%In other words, 
	%$\pi^\iota$ be the permutation defined by
	%\begin{equation}\label{eq:def_L}
	%\pi^\iota(i):=2\pi(n)-\pi(i) \qquad (1\le i\le n) .    
	%\end{equation}
	The involution $\iota$ preserves the set of arc permutations $\A_{n+2}\subset \S_{n+2}$,
	%is closed under the involution $\iota$ %left multiplication by $w_0$, 
	and commutes with the above $\tC_n$-action.
	%
	%\todo{YR: 
	%Should it be replaced by the longest element in $\S_{n+1}$?\\ Note that multiplication by $v$ is translated to left multiplication 
	%by the longest element in $\S_{n+1}$ (and not in $\S_{n+2}$),
	%and that $\A_{n+2}$ is invariant under this multiplication as well.\\
	%Note also that left multiplication by the longest element in $\S_{n+2}$ is negation after addition of 1 to the $b$ entry (i.e. translation)? does this element may replace $v$ as a generator
	%of $K_0$ (or its conjugates)?
	%}
	%
	%The set of arc permutations $\A_{n+2}\subset \S_{n+2}$ is
	%closed under left multiplication by $w_0$, and this multiplication commutes with the above $\tC_n$-action, namely,
	%
	%%furthermore, 
	%for every %$1\le i\le n$ and 
	%$\pi\in \A_{n+2}$
	%\[
	%\rho(s_i) ( w_0\cdot \pi) = w_0 \cdot \rho(s_i)(\pi)  
	%\qquad (\forall \pi\in \A_{n+2},\, 0 \le i \le n).
	%\]
	Hence, %the $\tC_{n}$-action on $\A_{n+2}$ determines 
	this yields
	a well-defined $\tC_{n}$-action on the set of equivalence classes $\A_{n+2}/\iota$.
	
	%\todo{YR: Is the map $\iota$ indeed the effect of multiplication by $v$?}
	
	\begin{theorem}\label{thm:main1}
		The $\tC_{n}$-module determined by the above action on $\iota$-equivalence classes of arc permutations 
		%$\A_{n+2}$ modulo the reverse operation is
		%the permutations 
		is multiplicity-free
	\end{theorem}
	
	Similar statements hold for $\tC_n$-actions on triangulations, words and trees;
	see Propositions ~\ref{prop:CTFT} and~\ref{prop:Cn-words} below.

	\medskip
	
	To prove Theorem~\ref{thm:main1} and the other statements, we introduce the notion of proto-Gelfand subgroups.
	%(to be defined now) for types $\tC_n$ and $\tB_n$.
	%, hereby to be defined.

	\begin{defn}
		Let $H\leq G$ be finite groups.
		The pair $(G,H)$ is called a {\em Gelfand pair} if
		the permutation representation of $G$ on the cosets of $H$ is multiplicity-free.  
		If $H\leq G$ are infinite groups, we say that $H$
		%$(G,H)$ 
		is a {\em proto-Gelfand subgroup} of $G$ if, 
		for every finite quotient $\varphi(G)$ of $G$, $(\varphi(G),\varphi(H))$ is a Gelfand pair.
	\end{defn}
	
	%\medskip
	
	We construct proto-Gelfand subgroups via a generalized flip action of $G$ %the affine Weyl group of type $\tC_n$
	on an infinite analogue of 
	the arc permutation set, 
	%$\Omega$%($\ZZ_2^n\times \ZZ$
	using affine permutations for an explicit presentation of the stabilizer.
	
	\medskip
	
	Let $G$ be a Weyl group of type $\tC_n$ ($n\geq 2$), with Coxeter generators $s_0, \ldots, s_n$, 
	and let %$X = \zzn$ where 
	$\ZZ_3=\{1,0,-1\}$.
	%Define a group homomorphism $\rho: G \to \Sym(\zzn)$%%%Rewritten by HP for comment#5%%%  
	% where $\ZZ_3=\{1,0,-1\}$.
	%determines a $G$-action (Observation~\ref{obs:relations} below).
	%\begin{defn}\label{def:c-action_general_new}
	%Define a group homomorphism $\rho: G \to \Sym(X)$ 
	%as a multiplicative extension of the following map:
	%by
	%\begin{equation*}
	%\rho(s_i)(a_1,a_2,\ldots,a_n,b) :=
	%    \begin{cases}
	%    (-a_1,a_2,\ldots,a_n,b),&\text{if } i=0;\\
	%    (a_1,\ldots,a_{i-1},a_{i+1},a_i,a_{i+2},\ldots,a_n,b),&\text{if } 0<i<n;\\
	%    (a_1,\ldots,-a_n,b+a_n), &\text{if } i=n,
	%    \end{cases}
	%\end{equation*}
	%for any $(a_1,a_2,\ldots,a_n,b) \in \zzn$.%%%Rewritten by HP for comment#5%%% 
	%\end{defn}
	
	%%%Rewritten by HP for comment#12%%% 
	\begin{defn}\label{def:c-action_general_new}
		Define a group homomorphism $\rho: G \to \Sym(\zzn)$%%%Rewritten by HP for comment#5%%%  
		%as a multiplicative extension of the following map:
		by
		\begin{equation*}
		\rho(s_i)(a_1,a_2,\ldots,a_n,b) :=
		\begin{cases}
		(-a_1,a_2,\ldots,a_n,b),&\text{if } i=0;\\
		(a_1,\ldots,a_{i-1},a_{i+1},a_i,a_{i+2},\ldots,a_n,b),&\text{if } 0<i<n;\\
		(a_1,\ldots,-a_n,b+a_n), &\text{if } i=n,
		\end{cases}
		\end{equation*}
		for any $(a_1,a_2,\ldots,a_n,b) \in \zzn$.%%%Rewritten by HP for comment#5%%% 
	\end{defn}
	
	\begin{observation}\label{obs:relations}
		$(\rho(s_i))_{i=0}^{n}$ satisfy all the defining relations of type $\tC_n$, so that 
		$\rho$ is %indeed 
		a well-defined group homomorphism. 
	\end{observation}
	
	The affine Weyl group of type $\tB_n$ is an index $2$ subgroup of $\tC_n$, and therefore also acts naturally on $\zzn$. Here is our main result.
	\begin{theorem}\label{thm:main12}
		Let $G$ be a Weyl group of type $\tB_n$ or $\tC_n$. For every $\omega\in \zzn$ whose stabilizer is not all of $G$ (namely, $\omega$ is not $G$-invariant), there exists a
		double cover of the stabilizer which is a proto-Gelfand subgroup of $G$. 
		%\todo{PH: If $w$ is $G$-invariant, then the stabilizer is $G$, which is also (proto-)Gelfand. So I would delete the clause ``which is not $G$-invariant'' from the statement.}
	\end{theorem}
	
	See Theorems~\ref{thm:main2_new} and~\ref{thm:main2_B} below. 
	The stabilizer itself is not proto-Gelfand;
	see Remark~\ref{rem:proto_HK}. See also Proposition~\ref{prop:m=2} for finite versions.
	
	\smallskip
	
	%The proof of Theorem~\ref{thm:main12} %and of the similar theorems goes as follows.
	% Let $G$ be an affine Weyl group (of type $\tC_{n}$ or $\tB_n$) that acts on a finite set $\Omega$. This gives a permutation representation, that is a 
	%$\varrho: G\rightarrow\Sym(\Omega)$. Let 
	%homomorphism $\overline{\phantom{G}}:G\rightarrow \overline{G}=G/\ker(\rho)$.
	%%denote the natural homomorphism. 
	%Let $K\le G$ denote a subgroup whose image is the stabilizer $\overline{K}=\Stab_{\overline{G}}(\omega)$ of an $\omega\in\Omega$. 
	%To show that $\Stab_{\overline{G}}(\omega)$ is a proto-Gelfand subgroup 
	%we apply 
	%To prove Theorem~\ref{thm:main12} we apply 
	The proof applies Gelfand's trick,     
	%\todo{RA: Define Galfand pair, Gelfand subgroup\\ in particular where $G$ is not compact}
	showing that, in a finite image,
	every double coset of the double cover is self-inverse. %$\overline{K}\overline{g}\overline{K}=\overline{K}\overline{g}^{-1}\overline{K}$. 
	%In fact, the corresponding (stronger) claim %$KgK=Kg^{-1}K$ 
	%is shown to hold in the group $G$ itself.  
	%It is then shown that a double-cover of the stabilizer is a Gelfand subgroup of $\tC_n$.
	%\\
	%YR: Clarify this pa.
	%}
	We work within the combinatorial realization of an affine Weyl group as a group of affine permutations, as in \cite[Chapter 8]{BB}. 
	%However, every care has been taken to 
	The subgroups involved are also described using generating sets.

	\bigskip
	
	The rest of the paper is organized as follows. 
	Various combinatorial flip actions of the groups of type $\tC_n$ are presented in Section~\ref{s:combinatorial_examples}.
	In Section~\ref{sec:Gelfand_gen} we generalize this setting and prove (the $\tC_n$ case of) Theorem~\ref{thm:main12}.
	The analogous result for type $\tB_n$ is proved in Section~\ref{sec:tB}. 
	Theorem~\ref{thm:main12} is applied in Section~\ref{sec:combin-revisited} to prove the multiplicity-freeness of permutation modules resulting from the previously introduced flip actions, see Figure~\ref{fig:overview} (in Section~\ref{sec:combin-revisited}) for a summary. 
	%as a reference.
	Final remarks and open problems are discussed in Section~\ref{sec:final}. In particular, we show in Proposition~\ref{prop:m=2} that the none of the original finite combinatorial actions is multiplicity-free. Thus the original Question~\ref{qu:original} is also answered.
	
	% other classical types.
	%}
	
	%%%%%%%%%%%%%%%%%%%%%%%%%%%%%%%%%%%%%%%%
	\section{Combinatorial flip actions}
	%of $\tC_n$}
	\label{s:combinatorial_examples}
	
	%In this section we present 
	Flip actions 
	on various combinatorial objects are presented in this section. 
	It will be shown %proved in later sections 
	that restrictions of these flip actions to distinguished subsets
	determine well defined $\tC_n$-actions.
	%, which are 
	One of the goals of this paper is to prove that %all
	these $\tC_n$-actions, modulo a natural involution, are 
	members of a wider family of multiplicity-free
	permutation modules.

	\subsection{Triangle-free triangulations
		%A %$\tC_{n}$-
		%flip action on 
		%Triangle-free triangulations
	}\label{ss:triangles}%\ \\
	
	%Label the vertices of a convex $n$-gon $P_{n}$ ($n > 4$) by the
	%elements $0,\ldots,n-1$ of the additive cyclic group $\ZZ_{n} = \ZZ / n\ZZ$.
	A {\em triangulation} (without extra vertices) 
	of a convex $n$-gon $P_{n}$, $n > 4$,
	is a set of $n-3$ non-crossing chords in $P_n$.
	The chords divide $P_n$ into $n-2$ triangles. 
	Each chord in a triangulation belongs to two adjacent triangles, whose union is a quadrangle. 
	Replacing the chord by the other diagonal of that quadrangle is a {\em flip} of the chord,
	yielding a different triangulation.
	The graph of all triangulations of a convex polygon, with edges corresponding to flips,
	was studied in a seminal paper by Sleator, Tarjan and Thurston~\cite{STT}.
	A partial action of the Thompson group on triangulations, by flips, was introduced by Dehornoy~\cite{Dehornoy}.
	A restriction of the flip action to a distinguished subset of triangulations was 
	considered in~\cite{AFR}, and will be described in the rest of this subsection.
	
	%In this subsection we %First,
	%recall basic concepts and main
	%results from~\cite{AFR}.
	%\subsection{Basic Concepts}\label{2.1}%\ \\
	
	%Label the vertices of a convex $n$-gon $P_{n}$ ($n > 4$) by the
	%elements $0,\ldots,n-1$ of the additive cyclic group $\ZZ_{n} = \ZZ / n\ZZ$.
	%Consider a triangulation (with no extra vertices) of the polygon.

	\medskip

	%\begin{defn}
	%\begin{itemize}
	%\item[1.] 
	Label the vertices of a convex $n$-gon $P_{n}$ ($n > 4$) by the
	elements $0,\ldots,n-1$ of the additive cyclic group $\ZZ_{n} = \ZZ / n\ZZ$.
	Each original edge of the polygon is called an {\em external edge} of the triangulation; 
	all other edges of the triangulation are called {\em internal edges}, or {\em chords}.
	%\item[2.]
	A triangulation of a convex $n$-gon $P_{n}$ is called
	{\em internal-triangle-free}, or simply {\em triangle-free},
	if it contains no triangle with $3$ internal
	edges. The set of all triangle-free triangulations of $P_{n}$ is
	denoted $TFT(n)$.
	%\end{itemize}
	%\end{defn}
	A chord in $P_{n}$ is called {\em short} if it connects the
	vertices labeled $i-1$ and $i+1$, for some $i\in \ZZ_{n}$. A
	triangulation is triangle-free if and only if it contains only two
	short chords~\cite[Claim 2.3]{AFR}.
	%A {\em proper coloring} of a triangulation $T\in TFT(n)$ is a
	%labeling of the chords by $0,\dots,n-4$ in the following inductive
	%way: Choose a short chord and label it $0$. Inductively, a chord
	%which was not yet labeled and is contained in a triangle whose
	%other chord has been labeled $i$, is labeled $i+1$. It is easy to
	%see that this uniquely defines the coloring.
	A {\em proper labelling} (or {\em orientation}, or {\em colouring}) of a triangulation $T\in TFT(n)$ is a
	labelling of the chords by $0,\ldots,n-4$ such that
	\begin{enumerate}
		\item
		One of the short chords is labelled $0$.
		\item
		If a triangle has exactly two internal edges then
		their labels are consecutive integers $i$, $i+1$.
	\end{enumerate}
	It is easy to see that each $T \in TFT(n)$ has exactly two proper colourings.
	The set of all properly labelled triangle-free triangulations is denoted $CTFT(n)$.
	%\end{defn} 
	In other words, a colored triangle-free triangulation in $CTFT(n)$ 
	%is determined
	may be identified with a sequence of $n-3$ non-intersecting diagonals in %the convex $n$-gon 
	$P_n$, 
	$(d_0,d_1,\dots,d_{n-4})$,  such that any two consecutive diagonals have a common vertex
	and the sequence starts and ends with a short chord.
	%Each chord in a triangulation is a diagonal of a unique quadrangle
	%(the union of two adjacent triangles). Replacing this chord by the
	%other diagonal of that quadrangle is a {\em flip} of the chord. 
	A flip in a colored triangulation preserves the colour of the flipped diagonal. Clearly, $|CTFT(n)|=n2^{n-4}$, as one short diagonal is freely chosen, each of the consecutive diagonals may stem from either end of the previous diagonal.
	%%%Rewritten by HP for comment#6%%% 

	The %affine Weyl
	group $\tC_{n}$ acts naturally on $CTFT(n+4)$ by flips:
	Each generator $s_i$ acts on each $T\in CTFT(n+4)$ by flipping the chord labelled $i$, provided that
	the result still belongs to $CTFT(n+4)$.
	If this is not the case then $T$ is unchanged by $s_i$.
	
	\begin{proposition}\label{t.action1}\cite[Proposition 3.2]{AFR}
		This defines a transitive $\tC_{n}$-action on $CTFT(n+4)$.
	\end{proposition}
	
	%This affine Weyl group action on $CTFT(n)$ was used to calculate the diameter of 
	%$\Gamma_n$.
	
	\begin{example}
		Let $T\in CTFT (8)$ be defined by the following sequence of $5$ non-intersecting diagonals of $P_8$:
		$(1,7),(1,6),(1,5),(2,5),(2,4)$. Then
		\[
		s_0 T = ( (6,8),(1,6),(1,5),(2,5),(2,4) ), \ \ \ \ 
		s_1 T = T, \ \ \ \ 
		s_2 T = ( (1,7),(1,6),(2,6),(2,5),(2,4) ), 
		\text{ etc.;}
		\]
		see Figure~\ref{fig:1}.
		
		\begin{figure}[htb]
			\begin{center}
				\begin{tikzpicture}[scale=0.9]
				\draw (-1,1.7) node {$T=$};
				\fill (2.4,3.4) circle (0.1) node[above]{\tiny 1}; \fill (3.4,2.4)
				circle (0.1) node[right]{\tiny 2}; \fill (3.4,1) circle (0.1)
				node[right]{\tiny 3}; \fill (2.4,0) circle (0.1) node[below]{\tiny
					4}; \fill (1,0) circle (0.1) node[below]{\tiny 5}; \fill (0,1)
				circle (0.1) node[left]{\tiny 6}; \fill (0,2.4) circle (0.1)
				node[left]{\tiny 7}; \fill (1,3.4) circle (0.1) node[above]{\tiny
					8};
				
				\draw
				(1,3.4)--(2.4,3.4)--(3.4,2.4)--(3.4,1)--(2.4,0)--(1,0)--(0,1)--(0,2.4)--(1,3.4);
				\draw(3.4,2.4)--(1,0);\draw (0,1)--(2.4,3.4); \draw[red]
				(2.4,3.4)--(1,0);\draw (3.4,2.4)--(2.4,0); \draw
				(0,2.4)--(2.4,3.4);
				%\draw[blue,dotted] (3.4,2.4)--(0,1);
				
				%\end{tikzpicture}
				%\end{center}
				%\end{figure}
				
				%\begin{figure}[hbt]
				%\begin{center}
				%\begin{tikzpicture}[scale=0.3]
				
				\draw (5,1.7) node {, \ $s_2 T=$};
				\fill (8.4,3.4) circle (0.1) node[above]{\tiny 1}; \fill (9.4,2.4)
				circle (0.1) node[right]{\tiny 2}; \fill (9.4,1) circle (0.1)
				node[right]{\tiny 3}; \fill (8.4,0) circle (0.1) node[below]{\tiny
					4}; \fill (7,0) circle (0.1) node[below]{\tiny 5}; \fill (6,1)
				circle (0.1) node[left]{\tiny 6}; \fill (6,2.4) circle (0.1)
				node[left]{\tiny 7}; \fill (7,3.4) circle (0.1) node[above]{\tiny
					8};
				
				\draw (7,3.4)--(8.4,3.4); 
				\draw (7,0)--(6,1)--(6,2.4)--(7,3.4);
				\draw(9.4,2.4)--(7,0);\draw (6,1)--(8.4,3.4); 
				\draw
				(8.4,0)--(9.4,1)--(9.4,2.4)--(8.4,3.4); 
				\draw (8.4,0)--(9.4,2.4);
				\draw (8.4,0)--(7,0); 
				\draw (6,2.4)--(8.4,3.4);
				\draw[red] (9.4,2.4)--(6,1);
				\end{tikzpicture}
			\end{center}
			\caption{$\tC_4$-action on $CTFT(8)$.} 
			\label{fig:1}
		\end{figure}
		
	\end{example}
	
	Let $\iota$ be the involution on the 
	set of triangle-free triangulation $CTFT(n+4)$ defined 
	%by a reflection 
	%along the 
	%line which crosses the angle of the
	%vertex labeled by $n+4$;
	as follows:
	for every $T\in CTFT(n+4)$
	%with a first short chord ($(i-1,i+1)$, 
	let $\iota(T)$ be the 
	colored triangle-free triangulation, 
	obtained via a reflection along the 
	line which crosses the angle of the
	vertex labelled by $n+4$.
	%polygon vertex between the endpints of the first short chord,
	See an example in Figure~\ref{fig:CTFT}.
	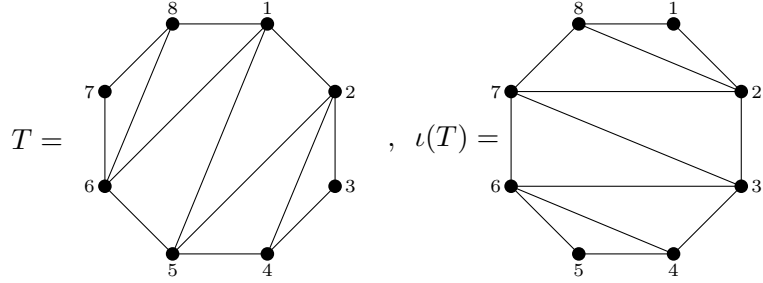
\begin{figure}[htb]
		\begin{center}
			\begin{tikzpicture}[scale=0.9]
			\draw (-1,1.7) node {$T=$};
			\fill (2.4,3.4) circle (0.1) node[above]{\tiny 1}; \fill (3.4,2.4)
			circle (0.1) node[right]{\tiny 2}; \fill (3.4,1) circle (0.1)
			node[right]{\tiny 3}; \fill (2.4,0) circle (0.1) node[below]{\tiny
				4}; \fill (1,0) circle (0.1) node[below]{\tiny 5}; \fill (0,1)
			circle (0.1) node[left]{\tiny 6}; \fill (0,2.4) circle (0.1)
			node[left]{\tiny 7}; \fill (1,3.4) circle (0.1) node[above]{\tiny
				8};
			
			\draw
			(1,3.4)--(2.4,3.4)--(3.4,2.4)--(3.4,1)--(2.4,0)--(1,0)--(0,1)--(0,2.4)--(1,3.4);
			\draw (0,1)--(1,3.4);
			\draw(3.4,2.4)--(1,0);\draw (0,1)--(2.4,3.4); \draw
			(2.4,3.4)--(1,0);\draw (3.4,2.4)--(2.4,0); 
			%\draw (0,2.4)--(2.4,3.4);
			%\draw[blue,dotted] (3.4,2.4)--(0,1);
			
			%\end{tikzpicture}
			%\end{center}
			%\end{figure}
			
			%\begin{figure}[hbt]
			%\begin{center}
			%\begin{tikzpicture}[scale=0.3]
			
			\draw (5,1.7) node {, \ $\iota(T) =$};
			\fill (8.4,3.4) circle (0.1) node[above]{\tiny 1}; 
			\fill (9.4,2.4)circle (0.1) node[right]{\tiny 2}; 
			\fill (9.4,1) circle (0.1) node[right]{\tiny 3}; 
			\fill (8.4,0) circle (0.1) node[below]{\tiny 4}; 
			\fill (7,0) circle (0.1) node[below]{\tiny 5}; 
			\fill (6,1) circle (0.1) node[left]{\tiny 6}; 
			\fill (6,2.4) circle (0.1) node[left]{\tiny 7}; 
			\fill (7,3.4) circle (0.1) node[above]{\tiny 8};
			
			\draw (7,3.4)--(8.4,3.4); 
			\draw (7,0)--(6,1)--(6,2.4)--(7,3.4);
			%\draw(9.4,2.4)--(7,0);
			%\draw (6,1)--(8.4,3.4); 
			\draw(9.4,1)--(6,1);
			\draw (6,2.4)--(9.4,1); 
			\draw (6,1)--(8.4,0); 
			\draw
			(8.4,0)--(9.4,1)--(9.4,2.4)--(8.4,3.4); 
			%\draw (8.4,0)--(9.4,2.4);
			\draw (8.4,0)--(7,0); 
			%\draw (6,2.4)--(8.4,3.4);
			\draw (7,3.4)--(9.4,2.4);
			\draw (6,2.4)--(9.4,2.4);
			%\draw(9.4,2.4)--(6,1);
			\end{tikzpicture}
		\end{center}
		\caption{The involution $\iota$ on $CTFT(8)$ is obtained by a reflection across the diagonal connecting vertices $8$ and $4$.} %with first short chord  $(1,7)$.} 
		\label{fig:CTFT}
	\end{figure}
	
	One can easily verify that for every $T\in CTFT(n+4)$ and $0\le i\le n$
	\[
	s_i(\iota T)=\iota (s_i T),
	\]
	Hence the $\tC_n$-action on 
	%$\iota$-congruence classes in $CTFT(n+4)$,
	the quotient $CTFT(n+4)/\iota$ is well defined.
	The following statement will be proved in Section~\ref{sec:combin-revisited}.
	
	\begin{proposition}\label{prop:CTFT}
		The $\tC_n$-action on $CTFT(n+4)/\iota$
		%,  modulo a reflection along the 
		%line which crosses the angle of the vertex $v_T$, 
		is multiplicity-free.
	\end{proposition}

	\subsection{%A %$\tC_{n}$-
		%flip action on a
		Arc permutations}\label{sec:arc}%\ \\
	
	%We begin with a typical example.
	%For a set $X$ let $\Sym(X)$ be the symmetric group of 
	%%bijections
	%permutations from $X$ to itself.
	%Let $\S_n:=\Sym([n])$ 
	%be the symmetric group on the set %letters 
	%$[n]:=\{1,\dots,n\}$.
	Denote a permutation $\pi\in \S_n$ by the sequence of values $[\pi(1),\dots,\pi(n)]$. Denote by $(i,j)$ the transposition interchanging $i$ and $j$. 
	%Let $0\le k\le n$.
	
	\medskip
	
	An {\em interval} in the cyclic group $\ZZ_n$ is a subset of the form
	%$[i,i+k]:=
	$\{i, i+1, \ldots, i+k\}$, where addition is modulo $n$.
	A permutation in the symmetric group $\S_n$ is an {\em arc
		permutation} if every suffix (equivalently prefix) forms an interval in $\ZZ_n$
	(where
	the letter $n$ is identified with zero). Denote by $\A_n$ the set of arc permutations in $\S_n$. Clearly, $|\A_n|=n2^{n-2}$.%%%Added by PH for comment#2%%%
	
	%\begin{example} T
	For example, the permutation %$[1,2,5,4,3]$ 
	$[3,4,5,2,1]$
	%$12543$
	%\todo{It should be $[1,2,5,4,3]$ as per the notation above.}
	is an arc permutation
	in $S_5$, but %$[1,2,5,4,3,6]$
	$[6,3,4,5,2,1]$
	is not an arc permutation in
	$S_6$, since $\{5,2,1\}$ is an interval in $\ZZ_5$ but not in
	$\ZZ_6$.
	%\end{example}

	\bigskip
	
	%We now generalize the 
	A $\tC_n$-action on arc permutations $\A_{n+2}$ was introduced in~\cite{TFT2}.
	Denote the adjacent transposition $(i,i+1)$  by $\sigma_i$.
	
	%A transposition $\sigma_i:=(i,i+1)$ acts on $\pi$ by switching the 
	%$i$-th and $i+1$-st entries.

	\begin{defn}\label{def:c-action_arc}
		For every $0\le i\le n$ and $\pi\in \A_{n+2}$, define %a homomorphism $\rho: \tC_{n} \to \Sym(\A_{n+2})$ 
		%as a multiplicative extension of the following map: by
		\[\rho^A(s_i)(\pi) :=
		\begin{cases}
		\pi \sigma_{i+1},&\text{if }\pi \sigma_{i+1}\in \A_{n+2};\\%,k};\\
		\pi, &\text{otherwise.}
		\end{cases}
		\]
	\end{defn}
	
	It % will be shown in Section~\ref{sec:combin-revisited}
	was shown in~\cite{TFT2} 
	that the above map determines a well defined transitive 
	$\tC_{n}$-action on $\A_{n+2}$. One of our goals is to show that this action,  modulo the involution $\iota$ defined in Equation~\eqref{eq:def_L},
	%a left multiplication by the longest element, 
	%natural inversion, 
	is multiplicity-free. This result will be proved and generalized to partial arc permutations in Subsection~\ref{sec:partial_arc}.

	\subsection{%A %$\tC_{n}$-flip action on Words
		Factorizations 
		%of the Coxeter element
		and geometric trees}%
	
	\subsubsection{Factorizations of the Coxeter element}\label{ss:words}
	%The affine Weyl 
	Recall the well known Hurwitz braid group  
	action on words~\cite{Hurwitz}.
	The braid group of type $A$ and rank $n-1$,  denoted ${\mathcal B}_n$,\footnote{not to be confused with the classical Weyl group of type $B$ and rank $n$, denoted $B_n$ in Subsection~\ref{sec:double_cover}.} 
		%on page \pageref{Bn-referenced-in-footnote}.}  %%%Rewritten by HP for comment#7%%%
	is generated by the set $\{b_1,\dots,b_{n-1}\}$,
	subject to the relations
	\[
	b_i b_j= b_j b_i \qquad \text{for}\ |j-i|>1,
	\]
	and
	\[
	b_i b_{i+1}b_i=b_{i+1} b_i b_{i+1}\qquad \text{for}\ 1 \le i < n.
	\]
	Let $G$ be a finite group, $C$ a conjugacy class in $G$,
	$g\in G$, and $n$ a positive integer. 
	Consider all factorizations of $g$ of length $n$, namely sequences (words) $w=(g_1,\dots,g_{n})$ of $n$ elements of $G$, such that $g_1g_2\cdots g_{n}=g$. 
	The Hurwitz action of the braid group ${\mathcal B}_n$ on this set of words is defined by 
	\begin{equation}\label{eq:braid}
	b_i(w):=(g_1,\ldots, g_{i-1},g_{i}g_{i+1}g_{i}^{-1}, g_i, g_{i+2},\ldots,g_{n}) \qquad (1 \le i \le n-1).
	\end{equation}
	Of special interest is the following example. %, see e.g.~\cite{??}.
	Let  $\{\sigma_i\mid 1\le i<n\}$ be the set of simple reflections (Coxeter generators) of the symmetric group $\S_n$ (Coxeter group of type $A_{n-1}$), and let 
	$\gamma_n := \sigma_1 \cdots \sigma_{n-1}$ be a Coxeter element in $\S_n$.
	We interpret $\sigma_i$ as the adjacent transposition $(i,i+1) \in \S_n$ and $\gamma_n$ as the $n$-cycle $(1,2,\dots,n) \in \S_n$. 
	Equation~\eqref{eq:braid} determines a braid group  ${\mathcal B}_n$-action on $F_{n+1}$,
	where $F_n$ is the set of all
	factorizations of $\gamma_n$ as a product of $n-1$ (not necessarily adjacent) transpositions.
	
	%Recall Hurwitz braid group action on words.
	
	\smallskip
	
	A factorization $t_1\cdots t_{n-1}$ of the $n$-cycle $\gamma_n = (1,2,\dots,n)$ as a product of transpositions is called {\em linear} if, for every $1 \le i \le n-2$, $t_i$ and $t_{i+1}$ have a common letter: one of them is $(a,b)$ and the other is $(c,b)$, where $a,b,c$ are distinct integers.
	Denote the set of linear factorizations of $\gamma_n$ by $LF_n$. %%%Rewritten by HP for comment#8%%%
	By \cite[Proposition~1.3]{YK}, $|LF_n|=n2^{n-3}$.%%%Rewritten by HP for comment#6%%% 

	Using the ${\mathcal B}_{n+2}$-action on $F_{n+3}$ define
	% For every $0\le i\le n$, define %a $\tC_{n}$-action on $LF_{n+3}$ 
	%as a multiplicative extension of the following map:
	the map $\rho^{LF}$ by
	\[\rho^{LF}(s_i)(w) :=
	\begin{cases}
	b_{i+1} (w) ,&\text{if } b_{i+1} (w)  \in LF_{n+3};\\
	b_{i+1}^{-1} (w) ,&\text{if } b_{i+1}^{-1} (w)\in LF_{n+3};\\
	w, &\text{otherwise.}
	\end{cases} \qquad (0\le i \le n,\, w \in LF_{n+3})
	\]
	\medskip
	
	Note that $b_i(w)\in LF_{n+3}$ and $b_i^{-1}(w)\in LF_{n+3}$
	are mutually exclusive conditions.
	
	\begin{example}
		%\noindent{\bf Example.}
		Let $w = \left( (2,3),\, (1,3), \,(3,5),\, (3,4) \right) \in LF_5$. 
		Then $b_1(w) \in LF_5$ (but $b_1^{-1}(w) \not\in LF_5$), 
		so that $\rho^{LF}(s_0)(w)= b_1(w) = \left( (1,2),\, (2,3),\, (3,5),\, (3,4) \right)$;  
		$b_2(w), b_2^{-1}(w) \not\in LF_5$, so that $\rho^{LF}(s_1)(w) = w$;
		and $b_3^{-1}(w) \in LF_5$, so that 
		$\rho^{LF}(s_2)(w) = b_3^{-1}(w) = \left( (2,3),\, (1,3),\, (3,4),\,(4,5) \right)$.
	\end{example}
	
	\medskip
	
	\begin{proposition}\label{prop:Cn-words}
		The map $\rho^{LF}$ determines a well-defined $\tC_n$-action on $LF_{n+3}$. Furthermore, 
		there exists an involution $\iota:LF_{n+3} \rightarrow LF_{n+3}$
		such that the induced $\tC_n$-action on $LF_{n+3}/\iota$ is well-defined and multiplicity-free.
	\end{proposition}
	
	This proposition will be proved in Section~\ref{sec:combin-revisited.factorizations}.

	%The latter $\tC_n$-action is isomorphic to 
	%the above $\tC_n$-action on geometric caterpillars,
	%see~\cite{YK}
	
	%\todo{This subsection should be completed}
	
	%\subsection{A $\tC_{n}$-action on caterpillars}\label{2.3}%\ \\
	
	%\bigskip
	
	\subsubsection{%A %$\tC_{n}$-
		%flip action on g
		Geometric caterpillars}\label{ss:trees}%\ \\
	
	The $\tC_n$-action on the set $LF_{n+3}$ described in previous subsection
	has a geometric interpretation.
	
	\smallskip
	
	%A {\em geometric caterpillar} of order $n$ is a special kind of 
	Consider a geometric tree of order $n$,
	whose vertices are drawn as points on a circle and labelled $0,1,\dots,n-1$ in clockwise order, 
	and whose edges are drawn as straight line segments which intersect only in common vertices.
	This tree is called a {\em geometric caterpillar} 
	%of order $n$} 
	if
	the set of internal vertices (non-leaves) forms a consecutive interval in $\ZZ_n$. This set is the {\em spine} of the caterpillar;
	the subgraph induced on it is a path.
	Denote by $GC_n$ the set of geometric caterpillars of order $n$. %%%Rewritten by HP for comment#8%%%
	
	%The path obtained by removing all leaves is called the \emph{spine} of the caterpillar.
	The edges of a geometric tree (in particular, a geometric caterpillar) are ordered by the {\em Goulden-Yong (GY) partial order}:
	%on the edges of a geometric non-crossing tree~\cite{Goulden-Yong}:
	%only linear extension of the following 
	%\begin{itemize}
	%    \item 
	the set of edges with a common vertex $i$ is linearly ordered in an anti-clockwise (cyclically decreasing) order of the other vertex, starting with $i-1$, see~\cite{Goulden-Yong}.
	%\item edges which do not share a common vertex must contain a vertex (or two) in the spine.
	%Order the edges by the clockwise order on the vertices in the spine.
	%\end{itemize}
	%For a geometric caterpillar, this order is linear.
	%Recall from~\cite{Goulden-Yong} that this order on the edges determines a bijection $\psi$ from the set of labeled geometric trees of order $n$ to the set of factorizations of the Coxeter element 
	%of the symmetric group $\S_n$, $\gamma=(1,2,\dots,n)$, 
	%as a product of $n-1$ transpositions.
	%Note that a 
	A geometric tree $T$ is a caterpillar if and only if the GY order on its edges is linear
	%, equivalently, if and only if $\psi(T)\in LF_n$
	~\cite[Theorem 3.2]{YK2}. Interpreting the edges in $\Gamma\in GC_n$ as transpositions
	in $\S_n$, the linear GY order induces a bijection $\psi:GC_n\rightarrow LF_n$~\cite[Cor.\ 4.4]{YK}.
	%Label the edges of $\Gamma_{n+3}$ by $0,\dots,n+2$ along the GY order.
	
	\medskip

	For a caterpillar $\Gamma\in GC_{n+3}$, 
	let %$[a_1,\dots,a_{n+3}]$ be a permutation in $S_{n+3}$, such that 
	$(e_0,\ldots,e_{n+1})$ be the sequence of edges of $\Gamma$,
	%a given geometric caterpillar, 
	listed according to the GY order. 
	For example, for the geometric caterpillar $\Gamma$ in Figure~\ref{fig:caterpillar},
	$e_0=(1,8),\ e_1=(1,7),\ e_2=(1,6),\ e_3=(1,5),\ e_4=(1,2),\ e_5=(2,4),\ e_6=(2,3)$. 
	
	For every $0\le i\le n$, 
	the edges $e_i$ and $e_{i+1}$ have a common vertex;
	let $a_i$ and $b_i$ be the vertices of $e_i$ and $e_{i+1}$ which are not common. 
	%$e_{i}=(a,b)$ and $e_{i+1}=(b,c)$
	Define $s_i\Gamma$ to be the geometric caterpillar whose set of edges is 
	obtained from the set of edges of $\Gamma$ by replacing
	either $e_i$ or $e_{i+1}$ by $(a,b)$  
	if the resulting tree is a geometric caterpillar; 
	and let $s_i\Gamma$ be $\Gamma$ otherwise. 
	One can verify that $s_i \Gamma$ is well defined.
	
	For example, consider $\Gamma$ in Figure~\ref{fig:caterpillar}.
	For $i=3$, $a=5$ and $b=2$. 
	Here $\Gamma\setminus e_4 \cup (2,5)$ is not a geometric caterpillar,
	but $\Gamma\setminus e_3 \cup (2,5)$ is;
	hence $s_3 \Gamma = \Gamma \setminus e_3 \cup (2,5)$.
	
	\medskip
	
	Define now, for each $0 \le i \le n$, a map
	$\rho^{GC}(s_i) : GC_{n+3} \to GC_{n+3}$ by
	%Furthermore, considering the mentioned above Goulden-Yong's bijection $\psi:GC_{n+3} \rightarrow LF_{n+3}$, 
	%from the set of labeled geometric trees of order $n$ to the set of linear factorizations of the Coxeter element
	\[
	\rho^{GC}(s_i)(\Gamma) := s_i\Gamma 
	\qquad (\forall\, \Gamma\in GC_{n+3}).
	\]
	The map $\rho^{GC}$ determines a well defined $\tC_n$-action on $GC_{n+3}$,
	which is isomorphic (via the bijection $\psi$) to the $\tC_n$-action on the set $LF_{n+3}$ described in the previous subsection.
	%It will be proved in Section~\ref{sec:combin-revisited} that this determines a well-defined $\tC_n$-action, see ?? below.
	%\todo{This subsection should be completed}

	\begin{figure}[htb]
		\begin{center}
			\begin{tikzpicture}[scale=0.9]
			\draw (-1,1.7) node {$\Gamma=$};
			\fill (2.4,3.4) circle (0.1) node[above]{\tiny 1}; \fill (3.4,2.4)
			circle (0.1) node[right]{\tiny 2}; \fill (3.4,1) circle (0.1)
			node[right]{\tiny 3}; \fill (2.4,0) circle (0.1) node[below]{\tiny
				4}; \fill (1,0) circle (0.1) node[below]{\tiny 5}; \fill (0,1)
			circle (0.1) node[left]{\tiny 6}; \fill (0,2.4) circle (0.1)
			node[left]{\tiny 7}; \fill (1,3.4) circle (0.1) node[above]{\tiny
				8};
			
			\draw (1,3.4)--(2.4,3.4)--(3.4,2.4)--(3.4,1);
			
			\draw[red] (2.4,3.4)--(1,0);\draw (3.4,2.4)--(2.4,0); \draw
			(0,2.4)--(2.4,3.4);\draw[blue,dotted] (3.4,2.4)--(1,0);\draw
			(0,1)--(2.4,3.4);
			%\end{tikzpicture}
			%\end{center}
			%\end{figure}
			
			%\begin{figure}[hbt]
			%\begin{center}
			%\begin{tikzpicture}[scale=0.3]
			\draw (4.8,1.7) node {, \ $s_3 \Gamma=$};
			\fill (8.4,3.4) circle (0.1) node[above]{\tiny 1}; \fill (9.4,2.4)
			circle (0.1) node[right]{\tiny 2}; \fill (9.4,1) circle (0.1)
			node[right]{\tiny 3}; \fill (8.4,0) circle (0.1) node[below]{\tiny
				4}; \fill (7,0) circle (0.1) node[below]{\tiny 5}; \fill (6,1)
			circle (0.1) node[left]{\tiny 6}; \fill (6,2.4) circle (0.1)
			node[left]{\tiny 7}; \fill (7,3.4) circle (0.1) node[above]{\tiny
				8};
			
			\draw (7,3.4)--(8.4,3.4)--(9.4,2.4)--(9.4,1); \draw
			(9.4,2.4)--(8.4,0); \draw (6,2.4)--(8.4,3.4);\draw[red]
			(9.4,2.4)--(7,0); \draw (6,1)--(8.4,3.4);
			\end{tikzpicture}
		\end{center}
		\caption{$\tC_5$-action on $GC_8$.} 
		\label{fig:caterpillar}
	\end{figure}

	%\end{frame}
	
	%\bigskip

	%%%%%%%%%%%%%%%%%%%%%%%%%%%%%%%%%%%%%%%%%%%
	\section{Proto-Gelfand pairs and generalized flip actions}%
	\label{sec:Gelfand_gen}

	\subsection{Gelfand and proto-Gelfand pairs}%
	\label{sec:proto}\ 
	
	%\todo{RA: The notion of Gelfand pair should be defined or cited.}
	
	\begin{defn}
		Let $H\leq G$ be finite groups.
		The pair $(G,H)$ is called a {\em Gelfand pair} if
		the permutation representation of $G$ on the cosets of $H$ is multiplicity-free.
	\end{defn}
	
	We will use the following lemma, sometimes called the {\em Gelfand trick}.
	%, which is an immediate consequence of~\cite[Theorem 45.2]{Bump}.
	
	%\todo{What about infinite $G$ ?
	%\cite[Theorem 45.2]{Bump}, which is applied in the following proof, should be stated here}
	
	\begin{lemma}\label{t:Gelfand_new}%
		\cite[Theorem 45.2]{Bump}
		Let $H \le G$ be finite groups,
		and suppose that there exists an involutive anti-automorphism $\iota$ of $G$ 
		such that for every $g\in G$, $\iota(HgH) = HgH$.
		Then $(G,H)$ is a Gelfand pair.
	\end{lemma}
	
	\begin{defn}
		Let $H\leq G$ be groups. We say that $(G,H)$ is a {\em proto-Gelfand pair} if, for every surjective homomorphism $\varphi$ of $G$ onto a finite group, $(\varphi(G),\varphi(H))$ is a Gelfand pair.%%%Rewritten by PH for comment#11%%%
	\end{defn}
	
	\begin{corollary}\label{Gelfand:cor_new}
		If every 
		%left 
		double coset of a subgroup $H$ of a group $G$ contains an involution, then $(G,H)$ is a proto-Gelfand pair.
	\end{corollary}
	
	%\todo{RA: Doesn't it suffice to prove that every double coset of $A$ in $G$ contains an involution?}
	
	\begin{proof}
		Suppose that $\varphi:G\rightarrow G_1$ is a homomorphism of $G$ onto a finite group $G_1$, 
		and let $H_1 := \varphi(H)$. 
		Take the inversion map $\iota(x) := x^{-1}$ $(\forall x \in G_1)$ as the anti-automorphism in Lemma~\ref{t:Gelfand_new}, 
		and pick images $\varphi(g)$ of involutions $g \in G$ as representatives of double cosets of $H_1$ in $G_1$. 
		Then, for every double coset, 
		\[
		\iota(H_1\varphi( g )H_1)= H_1\varphi(g^{-1})H_1=H_1\varphi(g)H_1.
		\]
		By Lemma~\ref{t:Gelfand_new}, $(G_1, H_1)$ is a Gelfand pair.
		
	\end{proof}

	%\todo{YR: to be added:
	\begin{corollary}\label{cor:free_action}
		Let $G$ be an infinite group acting transitively on a finite set $X$, and let $H$ be the stabilizer of some point of $X$. 
		If $(G,H)$ is a proto-Gefand pair, then the action of $G$ on $X$
		is multiplicity-free.
	\end{corollary}
	
	%\todo{PH: The $G$-action on $X$ is equivalent to the action on the cosets of the stabilizer $H$. This action maps $G$ into the finite group $\Sym(X)$. Since $(G,H)$ is proto-Gelfand, $(G/K,H/K)$ is a Gelfand pair, as required.}
	
	\begin{proof}
		The kernel $K$ of the $G$-action on $X$ is the intersection of the stabilizers of all points. Each stabilizer is of finite index in $G$, and therefore so is $K$.
		%conjugate to $H$, and is transitive, the stabilizer $H$ is of index $|X|<\infty$.
		%The kernel $K\le G$ of the $G$-action on $X$ is the intersection of $|X|$ stabilizers, hence of finite index $\le |X|^2$. 
		%Consider the natural homomorphism $\varphi:G \to G/K$.
		The %$G$-
		representation of $G$ on $G/H$ 
		%is isomorphic to the $\varphi(G)$-representation on $\varphi(G)/\varphi(H)$.
		factors through $G/K$.
		Since $(G,H)$ is proto-Gelfand, $(G/K,H/K)$ is a Gelfand pair, hence 
		the permutation representation of $G$ on $X$
		is multiplicity-free, as claimed.
		
	\end{proof}
	
	%}
	
	\subsection{A generalized flip action of $\tC_n$}%
	\label{sec:gen_flip_action}
	%{--- A planned replacement of Section 3}
	
	Let $G$ be an affine Weyl group of type $\tC_n$ ($n \ge 2$), with Coxeter generators $s_0, \ldots, s_n$.
	%Let %$X = \zzn$ where 
	%$\ZZ_3=\{1,0,-1\}$.%%%Rewritten by HP for comment#5%%% 
	%
	%\begin{defn}\label{def:c-action_general_new}
	%Define a group homomorphism $\rho: G \to \Sym(\zzn)$%%%Rewritten by HP for comment#5%%%  
	%as a multiplicative extension of the following map:
	%by
	%\begin{equation*}
	%\rho(s_i)(a_1,a_2,\ldots,a_n,b) :=
	%    \begin{cases}
	%    (-a_1,a_2,\ldots,a_n,b),&\text{if } i=0;\\
	%    (a_1,\ldots,a_{i-1},a_{i+1},a_i,a_{i+2},\ldots,a_n,b),&\text{if } 0<i<n;\\
	%    (a_1,\ldots,-a_n,b+a_n), &\text{if } i=n,
	%    \end{cases}
	%\end{equation*}
	%for any $(a_1,a_2,\ldots,a_n,b) \in \zzn$.%%%Rewritten by HP for comment#5%%% 
	%\end{defn}
	%
	%It is easy to see that 
	%$(\rho(s_i))_{i=0}^{n}$ satisfy the braid relations of type $\tC_n$, so that 
	%
	%\begin{observation}\label{obs:relations}
	%$(\rho(s_i))_{i=0}^{n}$ satisfy all the defining relations of type $\tC_n$, so that 
	%$\rho$ is %indeed 
	%a well-defined group homomorphism. 
	%\end{observation}
	
	The $G$-action of Definition~\ref{def:c-action_general_new} on the set %$X=%%%Rewritten by HP for comments#5 and #12%%% 
	$\zzn$ is not transitive. It has trivial single-point orbits of the form $\{(0,\ldots,0,b)\}$ for every integer $b$. 
	Its other orbits are $\Omega_{n,0},\,\Omega_{n,1},\ldots,\Omega_{n,n-1}$, where 
	\begin{equation}\label{eq:omega_def}
	\Omega_{n,k} :=
	\{(a_1,\ldots,,a_n,b) \mid k=|\{1 \le i \le n\mid a_i=0\}| \} 
	\qquad (0 \le k \le n-1).
	\end{equation}
	
	The goal of this section is to show that the action of $G$ on each of these orbits is ``almost'' multiplicity-free. 
	The actions in Section~\ref{s:combinatorial_examples} are isomorphic to $G$-actions on quotients of $\Omega_{n,0}$.
	
	\medskip
	
	\begin{defn}
		For $0\leq k\leq n-1$ let $H_k\leq G$ %\tC_n$
		denote the stabilizer of 
		\[
		\omega_k := (0,\ldots,0,1,\ldots,1,0) \in \Omega_{n,k}.
		\]
	\end{defn}
	%Define $c := s_0s_1\cdots s_n$, a Coxeter element in $G$,
	%and $g_0 := c s_{n-1} c^{-1}$. 
	%By inspection, $H_0$ contains $H := \langle s_1,\ldots,s_{n-1},g_0\rangle$. 
	In Subsection~\ref{sec:stabilizer} we find generators for $H_k$ and describe its structure explicitly.
	%show that $H_k \cong \tC_k \times \tA_{n-k-1}$. 
	In Subsection~\ref{sec:double_cover} we proceed to confirm that a certain double cover $K_k$ of $H_k$ is a proto-Gelfand subgroup of $G$.
	
	%\todo{RA: 
	%For $k = 0$, what is $\tC_0$?
	%For $k = n-1$, what is $\tA_0$?\\
	%YR: Apparently, $\tC_0=\tA_0$ is the trivial group.\\
	%PH: As a matter of fact, $\tC_1$ and $\tA_1$ are usually not defined, either. Both are isomorphic to the infinite dihedral group. I amended the Proposition, below.
	%}

	\subsection{Structure of the stabilizer}%
	\label{sec:stabilizer}
	
	Let $G$ be an affine Weyl group of type $\tC_n$.
	Recall from~\cite[Section 8.4]{BB} that $G$ can be described, combinatorially, as the group of odd $(2n+1)$-periodic bijections of $\ZZ$ onto itself. Explicitly, we identify $G$ with the group of all bijections $u:\ZZ \to \ZZ$ 
	%\[
	%\tC_n: \ZZ \mapsto \ZZ
	%\]
	satisfying
	\[
	u(-i)=-u(i) \quad (\forall i)\qquad \text{ and }\qquad
	u(i+2n+1)=u(i)+2n+1 \quad (\forall i).
	\]
	It follows that $u(0) = 0$, and that $u$ is determined by its values on the interval $[1,n]$. 
	We represent $u$ by the sequence $[u_1,u_2,\ldots,u_n]$, called the {\em window} of $u$, where $u_i := u(i)$.
	The simple reflections have the form
	\begin{equation*}
	s_i = 
	\begin{cases}
	[-1,2,3,\ldots,n],&\text{if } i=0;\\
	[1,\ldots,i-1,i+1,i,i+2,\ldots,n],&\text{if } 0<i<n;\\
	[1,\ldots,n-2,n-1,n+1],&\text{if } i=n.
	\end{cases} 
	\end{equation*}
	
	Now we describe explicitly the image of $\omega_k \in \zzn$ under $\rho(u)$ (see Definition~\ref{def:c-action_general_new}), 
	for an  arbitrary $u\in G$.
	%is an odd $(2n+1)$-periodic bijection of $\ZZ$.
	
	\begin{defn}\label{def:sign_new}
		For $0 \le k \le n-1$, 
		the {\em $k$-sign} of an integer $t \in \ZZ$
		%$t \not\equiv 0 \pmod{2n+1}$ 
		is
		\begin{equation*}
		\varepsilon_k(t) :=
		\begin{cases}
		1,&\text{if } t \equiv k+1, \ldots, n \pmod{2n+1};\\
		0,&\text{if } t\equiv -k, \ldots, k \pmod{2n+1};\\
		-1,&\text{if } t \equiv -n, \ldots, -(k+1) \pmod{2n+1}.
		\end{cases}
		\end{equation*}
	\end{defn}
	
	For any given $u \in G$ and any $t \not\equiv 0\pmod{2n+1}$, exactly one integer congruent to either $t$ or $-t$ $\pmod{2n+1}$ is in the window of $u$.
	We therefore have the following congruence for the sum of the elements in the window
	\[
	\sum_{j=1}^n \varepsilon_k(u_j)u_j \equiv \sum_{j=1}^n \varepsilon_k(j) j =\frac{(k+n+1)(n-k)}{2}
	\pmod{2n+1}.
	\]
	For $u\in G$ define %accordingly 
	\[
	P_k(u):=
	\frac{\sum_{j=1}^n \varepsilon_k(j) j - \sum_{j=1}^n \varepsilon_k(u^{-1}(j))u^{-1}(j)}{2n+1} \in \ZZ.
	\]
	
	\begin{lemma}\label{lem:window_to_rho_new}
		For any $u \in G$, %with window $[u_1, \ldots, u_n]$,
		the image of $\omega_k$ under $\rho(u)$ is the vector
		\[
		r_k(u) :=
		\left(\varepsilon_k(u^{-1}(1)),\ldots,\varepsilon_k(u^{-1}(n)),P_k(u)\right)\in\zzn.
		\]
		%\[\left(\frac{1+\varepsilon(u(1))}{2},\ldots,\frac{1+\varepsilon(u(n))}{2},\frac{\binom{n+1}{2}-\sum_{j=1}^n \varepsilon(u(j))u(j)}{2n+1}\right)\in\zn.
		%\]
	\end{lemma}
	
	\begin{proof}
		The proof is by induction on the Coxeter length of $u$.
		
		The claim is certainly true for the identity element $id_G = [1,2,\ldots,n]$ of $G$,
		since $r_k(id_G) = \omega_k$.
		It remains to show that if $\rho(u)(\omega_k) = r_k(u)$
		then $\rho(s_i u)(\omega_k) = r_k(s_i u)$
		for every $0 \le i \le n$. 
		
		For $0<i<n$ and $1\leq j\leq n$, $(s_i u)^{-1}(j) = u^{-1}(s_i(j))$ swapping the values $u^{-1}(i)$ and $u^{-1}(i+1)$. Clearly, $P_k(u)$ does not change. 
		Thus $r_k(s_i u)$ is obtained from $r_k(u)$ by swapping the $i$-th and $(i+1)$-st entries. 
		Also, $\rho(s_iu)(\omega_k)$ is obtained from $\rho(u)(\omega_k)$ by swapping the $i$-th and $(i+1)$-st entries. The claim thus holds.
		
		%\smallskip
		
		For $i = 0$, the first entry of $r_k(u)$ changes sign:
		\[
		\varepsilon_k((s_0 u)^{-1}(1)) 
		= \varepsilon_k(u^{-1}(-1)) 
		= -\varepsilon_k(u^{-1}(1)),
		\]
		as both $\varepsilon_k$ and $u^{-1}$ are odd functions. 
		On the other hand, $\varepsilon_k((s_0 u)^{-1}(j)) = \varepsilon_k(u^{-1}(j))$ for $2 \le j \le n$
		and also the last entry does not change:
		$P_k(s_0 u)=P_k(u)$, 
		%does not change 
		since 
		\[
		\varepsilon_k((s_0 u)^{-1}(1)) \cdot (s_0 u)^{-1}(1)
		= -\varepsilon_k(u^{-1}(1)) \cdot (-u^{-1}(1)) 
		= \varepsilon_k(u^{-1}(1)) \cdot u^{-1}(1). 
		\]
		Indeed, the only difference between $\rho(s_0u)(\omega_k)$ and $\rho(u)(\omega_k)$ is the negation of the first entry.
		
		%\smallskip
		
		Finally, for $i=n$,
		$s_n(n) = n+1 = -n + (2n+1)$ so that
		$s_n^{-1}(n) = -n + (2n+1)$ and
		\[
		\varepsilon_k((s_n u)^{-1}(n)) 
		= \varepsilon_k(u^{-1}(-n + (2n+1))) 
		= -\varepsilon_k(u^{-1}(n)).
		\]
		As for the last entry,
		%last entry changes by $\varepsilon_k(u^{-1}(n))$ as
		\begin{align*}
		(2n+1)(P_k(s_n u)-P_k(u))
		&=\varepsilon_k(u^{-1}(n))u^{-1}(n) - \varepsilon_k((s_n u)^{-1}(n))(s_n u)^{-1}(n)\\
		%\frac{\varepsilon(u(n))u(n)-\varepsilon(u(n+1))u(n+1)}{2n+1}=
		&=\varepsilon_k(u^{-1}(n))u^{-1}(n) + \varepsilon_k(u^{-1}(n))(2n+1-u^{-1}(n)) \\
		&= (2n+1)\varepsilon_k(u^{-1}(n)).
		\end{align*}
		Indeed, $\rho(s_n u)(\omega_k)$ is obtained from $\rho(u)(\omega_k)$ by negating the $n$-th entry and adding $\varepsilon_k(u^{-1}(n))$ to the last entry.
	\end{proof}

	\begin{defn}\label{def:exp}
		Let $m \in \ZZ$. If $m = a + (2n+1)b$ with $a \in [-n,n]$ and $b \in \ZZ$,
		define the {\em exponent} of $m$ to be $\lambda(m) := b$.
		We shall also employ the notation $a^{*b}=a+(2n+1)b$.%%%Changed the notation by PH for comment#14%%%
		%$\pm a \in [n]$ and $k \in \ZZ$ we employ the shorthand $a^{*+k}:=a+k(2n+1)$. We also write $\lambda(a^{*+k}):=\lambda(a+k(2n+1))=k$ for the exponent.
	\end{defn}
	%\todo{PH: Overlap with Definition~\ref{def:exp-sign}}
	%$a^{*-k}=a-k(2n+1)$, but ONLY in the window.
	%\todo{RA: I changed the definition of $g_k$. The new $g_k$ is the old $g_{k+1}$.}
	For example, $n+1=(-n)^{*+1}$. The exponential notation emphasises the transport of structure without reference to $n$: for $u\in G$ we have $u(a^{*b})=u(a)^{*b}$. This is comes handy in the proof of Proposition~\ref{t:H_k}.%%%Added by PH for comment#14%%%
	
	As promised, we now describe the stabilizer $H_k = \Stab_G(\omega_k)$.
	%in terms of the abstract generators of $G$. 
	Define the following elements of $G$, written as products of Coxeter generators and also using the window notation:
	%. Note again that the action is on the left so the products are resolved from-right-to-left. For each we provide abstract generators and the window as an affine permutation: (for $1\leq k\leq n$)
	\[
	\begin{array}{rlll}
	c
	&:= s_0s_1\cdots s_{n-1}s_n
	&= [2,3,\ldots,n-1,n,1^{*+1}];\\
	g_k 
	&:= c^{-1}s_1\cdots s_{k}s_{k+1}s_{k}\cdots s_1c
	&= [1,2,\ldots,k,n^{*-1},k+2,\ldots,n-1,(k+1)^{*+1}]&(0\leq k\leq n-2);\\
	h_k 
	&:= s_ks_{k+1}\cdots s_{n}\cdots s_{k+1}s_k
	&= [1,2,\ldots,k-1,(-k)^{*+1},k+1,\ldots,n]&(1\leq k\leq n-1).
	%;\\
	%x_k 
	%%&:= s_{k-1}\cdots s_1s_0s_1\cdots s_{n-1}s_ns_{n-1}\cdots s_k 
	%&:= s_{k-1}\cdots s_1 c s_{n-1}\cdots s_k 
	%&= [1,2,\ldots,k-1,k^{*+1},k+1,\ldots,n].
	\end{array}
	\]
	
	%\todo{RA:
	%What are $g_{n-1}$ and $h_0$?\\
	%PH: Undefined/meaningless. I set the bounds above %and amended the statement in the Proposition below.
	%}
	%\todo{YR: $x_k$ is not used in this subsection}
	
	%\todo{YR: %In the window of $c$, the letter $n$ is missing.\\
	%Are abstract definitions of these elements coherent with the windows?}
	
	% This also provides a natural map $\pi:G\mapsto \S_n$. 
	
	%\todo{YR: how? by letting $\pi(i):= j$ if $u_i=\pm j + k(2n+1)$ for some $k$?\\
	%It should be said that this map is a well defined homomorphism; a reference to the equivalent statement in~\cite[Section 8.4]{BB}] should be added.\\
	%Algebraically, the map is an epimorphism defined by mapping $s_0$ and $S_n$ to 1.}
	
	%Let $\psi:\tC_n \mapsto B_n$ be the group homomorphism defined by
	%\[
	%\psi(s_i):=
	%    \begin{cases}
	%    s_i, & i\ne 0;\\
	%    1, & i=n.
	%    \end{cases}
	%\]
	%Equivalently, for every $u \in \tC_n$ and $1\le i\le n$, $\psi(u)(i):= j$ if $u_i = j^{*+\ell}$ for some $\ell \in \ZZ$ and $-n \le j \le n$.
	
	%\todo{RA: The mapping $\psi$ is not used in the updated proof of Proposition~\ref{t:H_k}, and its definition here was commented out.}
	
	\begin{proposition}\label{t:H_k}
		For $0 \le k \le n-1$,
		the stabilizer of $\omega_k$ under the action of $G$ is 
		\[
		H_k 
		= \langle s_0,\ldots,s_{k-1},h_k,g_k,s_{k+1},\ldots,s_{n-1}\rangle
		= H_k^L \times H_k^U,
		\]
		where 
		{\rm
			\[
			H_k^L := \langle s_0,\ldots,s_{k-1},h_k \rangle \cong \tC_k \text{ for } k>0 
			\text{ (while $H_0^L = 1$),}
			\]
		}
		and
		{\rm
			\[
			H_k^U := \langle g_k,s_{k+1},\ldots,s_{n-1} \rangle 
			\cong \tA_{n-k-1} \text{ for } k<n-1
			\text{ (while $H_{n-1}^U = 1$).}
			\]
		}
		Note that $\tC_1 \cong \tA_1 \cong \langle s_0,s_1 \,|\, s_0^2 = s_1^2 = 1 \rangle$, the infinite dihedral group.
		
		Explicitly, in window notation,
		\[
		H_k^L = 
		\{[a_1^{*b_1}, \ldots, a_k^{*b_k}, k+1, \ldots, n] \mid
		(|a_1|, \ldots, |a_k|) \in \S_k,\, 
		b_1, \ldots, b_k \in \ZZ\}
		\]
		while
		\[
		H_k^U = 
		\{[1, \ldots, k, a_{k+1}^{*b_{k+1}}, \ldots, a_n^{*b_n}] \mid
		(a_{k+1}-k, \ldots, a_n-k) \in \S_{n-k},\, b_{k+1}, \ldots, b_n \in \ZZ, b_{k+1} + \ldots + b_n = 0\}.
		\]
	\end{proposition}
	\begin{proof} 
		%We will show that the stabilizer of $\omega_k$, $H_k$, is equal to $H$.
		By Lemma~\ref{lem:window_to_rho_new}, the stabilizer $H_k$ consists of those $u\in G$ for which 
		\begin{equation}\label{eq:st1}
		\varepsilon_k(u^{-1}(i)) = 
		\begin{cases}
		0, &\text{if } 1\le i\le k;\\
		1, &\text{if } k<i\le n
		\end{cases}
		\end{equation}
		and 
		\[
		\sum_{j=1}^n \varepsilon_k(u^{-1}(j))u^{-1}(j)
		= \sum_{j=1}^n \varepsilon_k(j) j.
		\]
		Given equations~\eqref{eq:st1}, the latter equation is equivalent to
		\begin{equation}\label{eq:st2}
		\sum_{j=k+1}^{n} \lambda(u^{-1}(j)) = 0.
		\end{equation}
		%see Definition~\ref{def:exp}.
		
		Denote 
		\[
		L := \{-k,\ldots,-1\} \cup \{1,\ldots,k\} 
		\qquad \text{(signed lower set)} 
		\]
		and 
		\[
		U := \{k+1,\ldots,n\} 
		\qquad \text{(positive upper set)}.
		\]
		Using the additional notations
		$\ZZ_0 := (2n+1)\ZZ$,
		$\ZZ_L := L + (2n+1)\ZZ$, $\ZZ_U := U + (2n+1)\ZZ$, 
		and $\ZZ_{-U} := (-U) + (2n+1)\ZZ$,
		it is clear that
		$\ZZ = \ZZ_0 \cup \ZZ_L \cup \ZZ_U \cup \ZZ_{-U}$ 
		is a disjoint union.
		
		Now let $u \in H_k$.
		By equations~\eqref{eq:st1}, $u^{-1}$ maps each of the sets $\ZZ_L$, $\ZZ_U$, and $\ZZ_{-U}$ into (thus onto) itself, and of course fixes $\ZZ_0$ pointwise.
		The same thus holds for $u$.
		Define $u^L: \ZZ \to \ZZ$ and $u^U: \ZZ \to \ZZ$ by
		\[
		u^L(i) := 
		\begin{cases}
		u(i), &\text{if } i \in \ZZ_L;\\
		i, &\text{if } i \in \ZZ_0 \cup \ZZ_U \cup \ZZ_{-U}
		\end{cases}
		\]
		and
		\[
		u^U(i) := 
		\begin{cases}
		u(i), &\text{if } i \in \ZZ_U \cup \ZZ_{-U};\\
		i, &\text{if } i \in \ZZ_0 \cup \ZZ_L.
		\end{cases}
		\]
		Denote $H_k^L := \{u^L \mid u \in H_k\}$
		and $H_k^U := \{u^U \mid u \in H_k\}$.
		Then clearly $H_k^L$ and $H_k^U$ are subgroups of $H_k$, and in fact $H_k = H_k^L \times H_k^U$.
		It remains to determine the structure of $H_k^L$ and $H_k^U$.
		
		Equation~\eqref{eq:st2} is not relevant to the elements of $H_k^L$. By equations~\eqref{eq:st1}, the elements of $H_k^L$ are all the odd $(2n+1)$-periodic bijections of $\ZZ$ onto itself which fix $\ZZ_0 \cup \ZZ_U \cup \ZZ_{-U}$ pointwise.
		Identifying $\ZZ_0 \cup \ZZ_L$ with $\ZZ$ by mapping $a + (2n+1)b$ to $a + (2k+1)b$
		(for $-k \le a \le k, b \in \ZZ$), we can view $H_k^L$ as the group of all odd $(2k+1)$-periodic bijections of $\ZZ$ onto itself, namely (for $k \ge 1$) the affine Weyl group of type $\tC_k$; for $k = 0$ it is, of course, the trivial group.
		Using the window notation, restricted to $\{1, \ldots, k\}$,
		we can identify the Coxeter generators $s_0^\prime,\ldots,s_{k-1}^\prime,s_k^\prime$ of $\tC_k$ 
		with the generators $s_0,\ldots,s_{k-1},h_k$ of $H_k^L$. 
		Note that $s_k^\prime(k) = (-k)^{*+1} = -k+(2k+1) \in \ZZ$ is identified with $h_k(k) = (-k)^{*+1} = -k + (2n+1) \in \ZZ_L$.
		Thus $H_k^L = \langle s_0,\ldots,s_{k-1},h_k \rangle$
		for $1 \le k \le n-1$, as claimed.
		
		As for $H_k^U$: By equations~\eqref{eq:st1} and~\eqref{eq:st2}, the elements of $H_k^U$ are all the odd $(2n+1)$-periodic bijections $u$ of $\ZZ$ onto itself which fix $\ZZ_0 \cup \ZZ_L$ pointwise, map $\ZZ_U$ onto itself, and satisfy
		\[
		\sum_{j \in U} \lambda(u(j)) = 0.
		\]
		Identifying $\ZZ_U$ with $\ZZ$ by mapping $j = a + (2n+1)b$ to $i = a-k + (n-k)b$
		(for $k+1 \le a \le n, b \in \ZZ$), we can view $H_k^U$ as the group of all $(n-k)$-periodic,
		not necessarily odd,
		bijections $v$ of $\ZZ$ onto itself satisfying
		\[
		\sum_{i = 1}^{n-k} v(i) = \sum_{i = 1}^{n-k} i \, .
		\]
		%\todo{YR: To be added:\\
		Recalling from~\cite[Section 8.3]{BB} that $\tA_{n-1}$ can be described as the group of $n$-periodic bijections $u$ of $\ZZ$ onto itself satisfying
		\begin{equation}\label{window_sum}
		\sum_{i=1}^{n} u(i) = \sum_{i=1}^{n} i = \binom{n+1}{2},
		\end{equation}
		%As before, we represent each $u \in \tA_{n-1}$ uniquely by the window 
		%$[u_1,u_2,\ldots,u_n]$, where $u_i := u(i)$. 
		we deduce that
		%}
		for $k \le n-2$, this is exactly the affine Weyl group of type $\tA_{n-k-1}$.
		%(see Subsection~\ref{An-action} ???).
		Using the window notation of $u$,
		restricted to $\{k+1, \ldots, n\}$ and shifted down by $k$,
		we can identify the Coxeter generators $s_0^\prime,s_1^\prime,\ldots,s_{n-k-1}^\prime$ of $\tA_{n-k-1}$ 
		with the generators $g_k,s_{k+1},\ldots,s_{n-1}$ of $H_k^U$. 
		For example, 
		$s_0^\prime = [(n-k)^{*-1},2,\ldots,n-k-1,1^{*+1}] \in \tA_{n-k-1}$ is identified with $g_k = [1,\ldots,k,n^{*-1},k+2,\ldots,n-1,(k+1)^{*+1}] \in H_k^U$.
		Thus $H_k^U = \langle g_{k+1},s_{k+1},\ldots,s_{n-1} \rangle$, which completes the proof.

	\end{proof}
	
	%\medskip This might be needed later:
	%\begin{equation*}
	%x_0=cs_{n}c^{-1}s_0 =[1^{*+1},2,\ldots,n].
	%\end{equation*}
	
	%\begin{equation*}
	%c(j)=s_0s_1\cdots s_n(j) =
	%\begin{cases}
	%j+n+2,&\text{ if }j\equiv n\pmod{2n+1};\\
	%j+1,&\text{ if }j\equiv 1,\ldots, n-1\pmod{2n+1};\\
	%j-n-2,&\text{ if }j\equiv -n\pmod{2n+1};\\
	%j-1,&\text{ if }j\equiv -1,\ldots, -n+1\pmod{2n+1}.
	%\end{cases} 
	%\end{equation*}
	%So \begin{equation*}
	%g_0(j)=cs_{n-1}c^{-1}(j) =
	%\begin{cases}
	%j+n+2,&\text{ if }j\equiv n,-1\pmod{2n+1};\\
	%j-n-2,&\text{ if }j\equiv 1,- n\pmod{2n+1};\\
	%j,&\text{ if }|j|\not\equiv 1,n\pmod{2n+1}.
	%\end{cases} 
	%\end{equation*}
	%This implies that the window of $ug_0$ is %$[u(n)-2n-1,u(2),\ldots,u(n-1),u(1)+2n+1].$
	
	%Similarly, \begin{equation*}
	%x_0(j)=cs_{n}c^{-1}s_0(j) =
	%\begin{cases}
	%j+2n+1,&\text{ if }j\equiv 1\pmod{2n+1};\\
	%j-2n-1,&\text{ if }j\equiv -1\pmod{2n+1};\\
	%j,&\text{ if }|j|\not\equiv 1\pmod{2n+1}.
	%\end{cases} 
	%\end{equation*}

	\subsection{%Gelfand pairs
		A double cover}%
	\label{sec:double_cover}
	
	We continue with $G = \langle s_0,s_1,\ldots,s_n \rangle$, an affine Weyl group of type $\tC_n$. 
	For $1 \le i \le n$ define
	\begin{equation}\label{eq:translation}
	%\begin{array}{rll}
	x_i
	:= s_{i-1} \cdots s_1 s_0 s_1 \cdots s_{i-1} s_i \cdots s_{n-1} s_n s_{n-1} \cdots s_i
	= [1,2,\ldots,i-1,i^{*+1},i+1,\ldots,n].
	%\end{array}
	\end{equation}
	These elements of $G$ generate
	\[
	X := \langle x_1,x_2,\ldots,x_n\rangle \cong \ZZ^n,
	\]
	the normal abelian subgroup of $G$ consisting of translations. 
	The group $G$ is a semidirect product $X \rtimes B$ 
	of the normal subgroup $X$ and
	the parabolic subgroup $B = \langle s_0,\ldots, s_{n-1} \rangle \leq G$, a finite Weyl group of type $B_n$\label{Bn-referenced-in-footnote}.
	%$X=\langle x_1,x_2,\ldots,x_n\rangle$. 
	In turn, $B$ is the semidirect product $V \rtimes S$
	of its parabolic subgroup $S = \langle s_1,\ldots, s_{n-1} \rangle$,
	isomorphic to the symmetric group $\S_n$,
	and the normal (in $B$) abelian subgroup $V = \langle e_1,e_2,\ldots,e_{n} \rangle \cong (\ZZ/2\ZZ)^n$, where 
	\[
	e_1 := s_0 = [-1, 2, \ldots, n]
	\]
	and, using the notation $g^h := h^{-1}gh$,
	\[
	e_{i+1} := e_{i}^{s_i} 
	= s_i \cdots s_1s_0s_1 \cdots s_i
	= [1, \ldots, i, -(i+1), i+2, \ldots, n] \qquad (i=1, \ldots, n-1). 
	\]
	The action of $S$ on $V$ is the permutation action:
	For $1 \le i \le n$ and $1 \le j \le n-1$,
	\begin{equation}
	e_i^{s_j} =
	\begin{cases}
	e_{i+1}& \text{if } i=j;\\
	e_{i-1}&\text{if } i=j+1;\\
	e_{i}&\text{otherwise.}
	\end{cases}
	\end{equation}
	
	%Let us put $v = e_1e_2\cdots e_{n} = [-1,-2,\ldots,-n] \in C_V(\S_n)$.
	We also have 
	\begin{equation}\label{x_to_e_new}
	x_i^{e_j} =
	\begin{cases}
	x_i^{-1}, &\text{ if } i=j \text; \\
	x_i, &\text{otherwise.}
	\end{cases}
	\end{equation}
	%\todo{YR: Should it be "if $i=j$" and not if "if $i=j+1$"?}
	
	%\todo{YR: Should notation $h^g:=ghg^{-1}$ be explained?} 
	
	%In fact, $x_1^{s_i}=x_1$ for every $i\geq 2$.
	Recall that for $0\leq k\leq n-1$ we have 
	$H_k = \langle s_0,\ldots,s_{k-1}, h_k \rangle \times \langle g_k,s_{k+1},\ldots s_{n-1} \rangle \cong \tC_k \times \widetilde{A}_{n-k-1}$.
	The corresponding translation subgroups are 
	$\langle x_i \,:\, 1 \leq i \leq k \rangle$ and
	$\langle x_ix_{i+1}^{-1} \,:\, k+1 \leq i \leq n-1 \rangle$.
	%
	%\todo{Old version: Consequently, $H_k X =\langle H_k, x_n \rangle=\langle x_n \rangle H_k$. Call this subgroup $L_k$. 
	%The factor group $L_k/H_k$ is cyclic of infinite order, generated by the image of $x_n$. 
	%}
	
	\begin{defn}\label{defn:T}
		For a subset $J \subseteq [n]$ of size $k$,
		let $L_J := [k] \setminus J$ 
		and $U_J := J \setminus [k]$. 
		Let $\tau_J \in \S_n$ be the involution exchanging the $i$-th smallest element of $L_J$ and the $i$-th smallest element of $U_J$, for all values of $i$. 
		Denote
		\[
		T:=\{\tau_J \mid J \subseteq [n],\, |J|=k\} \subseteq \S_n.
		\]
	\end{defn}
	
	For example, if $J=\{2,3,6,8\} \subseteq [9]$ (thus $n=9$ and $k=4$), then 
	$L_J = \{1,4\}$, $U_J = \{6,8\}$ and 
	$\tau_J = (1,6)(4,8)$. 
	Note that $\tau_J$ maps $[k]$ onto $J$, and vice versa.
	
	Consider the natural embeddings of $\S_n$ and of $\tC_k \times \tC_{n-k}$ into $\tC_n$, using the window notation.
	Indeed, as in Proposition~\ref{t:H_k}, the two factors are
	\[\begin{array}{l}
	\left\{[a_1^{*b_1}, \ldots, a_k^{*b_k}, k+1, \ldots, n] \mid
	(|a_1|, \ldots, |a_k|) \in \S_k,\, 
	b_1, \ldots, b_k \in \ZZ\right\}\text{ and}\\
	\left\{[1,2,\ldots, k, a_{k+1}^{*b_{k+1}}, \ldots, a_n^{*b_n}] \mid
	(|a_{k+1}|-k, \ldots, |a_n|-k) \in \S_{n-k},\, 
	b_{k+1}, \ldots, b_n \in \ZZ\right\}.
	\end{array}\]
	%%% Explicit version added by PH for comment #16%%%
	\begin{observation}\label{obs:transversal}
		The set $T$,
		consisting of $\binom{n}{k}$ involutions,
		is a transversal (i.e., a set of coset representatives) for $\tC_k\times \tC_{n-k}$ in $\tC_n$.
		%(indeed, swaps between the lower and upper sets $L$ and $U$).
	\end{observation}
	
	%\todo{YR: To be clarified; the elements in $T$ are indexed by subsets $J\subseteq [n]$ of size $k$, where $\tau_J$ is a product of disjoint transpositions switching the $i$-th element in $J\setminus [k]$ with the $i$-th element in $[k]\setminus J^c$. For example,letting $n=9, k=4$ and $J=\{2,3,6,8\}$, $\tau_J=(1,6)(4,8)$.}
	
	%\todo{Old version: 
	%Then $G=TVL_k$.\\
	%
	%YR: Since $G = BX = VSX = VTH_k X = TVL_k$.
	%Should it be explained?
	%} 
	
	\medskip
	
	Finally, define the following double cover of the stabilizer $H_k$. 
	
	\begin{defn}\label{defn:double_cover}
		Let
		\[
		v := 
		e_1e_2\cdots e_{n} = [-1,-2,\ldots,-n] \in V,
		%C_V(S)
		\] 
		the longest element 
		in the natural embedding of $B_n$ in $\tC_{n}$,
		and let
		\[
		K_k := \langle H_k,v \rangle.
		\]
	\end{defn}
	%As $v\in C_{V}(\S_n)$, and $c=x_0t$ for $t=s_1\cdots s_{n-1}\in \S_n$, we have
	%\[g_0^v=(x_0ts_{n-1}t^{-1}x_0^{-1})^v=x_0^{-1}ts_{n-1}t^{-1}x_0\in H,
	%\]
	%Hence
	Note that $|K_k:H_k|=2$. % and $H\langle D=LK=N_G(H)$. 

	\begin{theorem}\label{thm:main2_new}
		$(G,K_k)$ is a proto-Gelfand pair for every $0\leq k\leq n-1$.
		%\todo{PH: We need a new name for this phenomenon. ``Proto-Gelfand pair'' or something similar. The meaning is: the double cosets are self-inverse (or the more general version), hence all finite quotients $(\varphi(G),\varphi(K))$ are Gelfand pairs.}
		%The action of $G$ on the cosets of $K$ is multiplicity-free.
	\end{theorem}

	\begin{proof}%[Proof of Theorem~\ref{thm:main2_new}]
		%hence $|K:H|=2$. % and $H\langle D=LK=N_G(H)$. 
		By Corollary~\ref{Gelfand:cor_new}, it suffices to show that every double coset of $K_k$ in $G$ contains an involution. A fortiori, it suffices to show this for every left coset $gK_k$.
		
		%\todo{ Old version: As $G=TVL_k=TV\langle x_n\rangle K_k=T\langle x_n\rangle VK_k$, every coset is of form $\tau x_n^m wK_k$ for some $\tau\in T$, $w\in V$ and integer $m$.
		%We can also assume that the expansion of $w$ in the natural basis of $V$ does not involve $e_j$ for $j\leq k-1$, as these are in $H_k$. \\
		%YR: Should it be $j\leq k$?
		%}
		
		Denote
		\[
		V_L := \langle e_1, \ldots, e_k \rangle
		\qquad \text{and} \qquad
		V_U := \langle e_{k+1}, \ldots, e_n \rangle,
		\]
		so that
		\[
		V = V_L \times V_U.
		\]
		Consider the natural embeddings $\tA_{n-k-1} \le \tC_{n-k}$ and $\tC_k \times \tC_{n-k} \le \tC_n$.
		Note that $V_L\le \tC_k$ and $V_U \le \tC_{n-k}$, and also
		$\tC_k = V_L \tC_k$ (trivially) and
		$\tC_{n-k} = \langle x_n\rangle V_U \tA_{n-k-1}$.
		%\begin{observation}\label{obs:C}
		%$\tC_k = V_L \tC_k$ and
		%$\tC_{n-k} = \langle x_n\rangle V_U \tA_{n-k-1}$.
		%\end{observation}
		Combining %Observation~\ref{obs:C} 
		this with Observation~\ref{obs:transversal} and Proposition~\ref{t:H_k}, 
		one deduces that  
		\[
		\begin{aligned}
		G 
		= \tC_n 
		&= T (\tC_k \times \tC_{n-k})
		= T (V_L \tC_k \times \langle x_n \rangle V_U \tA_{n-k-1})
		= T \langle x_n \rangle V (\tC_k\times \tA_{n-k-1}) \\
		&= T \langle x_n \rangle V H_k
		= T \langle x_n \rangle V \langle v \rangle H_k= T\langle x_n\rangle V K_k.
		\end{aligned}
		\]
		Thus every left coset of $K_k$ in $G$ is of the form
		$\tau x_n^d w K_k$ for some $\tau \in T$, $w \in V$ and integer $d$.
		Note that $V \cap H_k = V_L$ and $V \cap K_k = V_L \langle v \rangle$.
		In the expression of $\tau$ as a product of disjoint cycles, each $2$-cycle contains one element of $[k]$ and one element of $[n] \setminus [k]$. It follows that there exists an element $w' \in V \cap K_k$ such that
		$ww'$ commutes with $\tau$ and $x_n^{ww'} = x_n^{-1}$.
		
		If $\tau(n) \ne n$ then $\tau(n)\le k$, hence
		$x_{\tau(n)}\in H_k$. 
		The coset representative $\tau x_n^d w w^\prime x_{\tau(n)}^{-d} \in \tau x_n^d w K_k$ is then an involution:
		\[
		%\begin{align*}
		(\tau  x_n^d ww' x_{\tau(n)}^{-d})^2
		= x_{\tau(n)}^d ww' x_n^{-d}x_n^{d}ww' x_{\tau(n)}^{-d}
		= 1.
		%\end{align*}
		\]
		On the other hand, if $\tau(n) = n$ then %$x_n^w=x_n^{w^\tau}=x_{\tau(n)}^{w^\tau}$ and 
		$\tau x_n^d ww' \in \tau  x_n^d w K_k$ is an involution:
		\[
		%\begin{align*}
		(\tau x_{n}^d ww')^2
		= x_{n}^d ww' x_n^{d}ww'
		= x_{n}^d(x_{n}^d)^{ww'}
		= 1.
		%\end{align*}
		\]
		
	\end{proof}
	
	An explicit description
	of the involutive $K_k$-coset representatives is provided in the following remark. This description may help the reader
	to follow the  proof of Theorem~\ref{thm:main2_new} above.

	%\todo{
	\begin{remark}\label{rem:explicit}
		One can apply the window notation %description  
		to get an explicit  description of a complete list of involutive $K_k$ left-coset representatives.
		Fix $0 \le k \le n-1$. Then
		%\[
		%T = T_0 \cup \cdots \cup T_k
		%\]
		%where, for $0 \le t \le k$,
		\[
		T :=
		\{(i_1,j_1)\cdots (i_t,j_t) \mid 
		t \ge 0,\, 1\le i_1 < \cdots <i_t \le k \text{ and } k+1 \le j_1 < \cdots < j_t \le n\}.
		\]
		%In particular, $T_0$ contains only the identity element.
		
		\smallskip
		
		The complete list of involutive $K_k$-coset representatives is a disjoint union of subsets indexed by $\tau\in T$:
		\[
		R_k := \bigsqcup_{\tau\in T} R_{k,\tau}
		\]
		%, where the affine permutations in these subsets are defined by the following window:
		%\[
		%\{[\pm 1, \pm 2, \dots, \pm (n-1),((-1)^{*d+1}n)^{*+d}]\mid d\in \ZZ\},
		%\]
		%(for $t=0$);
		where, for every $\tau=(i_1,j_1)\cdots (i_t,j_t) \in T$,
		%associate 
		$R_{k,\tau}$ is the set of affine permutations  $\sigma=[\sigma(1),\ldots,\sigma(n)]$, in window notation, satisfying: 
		\begin{itemize}
			\item
			$\sigma(i)=i$  for all $ i\in [k]\setminus\{i_1,\dots,i_t\}$;
			\item
			$\sigma(i)=\pm i$ for all $k<i\not\in \{j_1,\dots,j_t,n\}$;
			\item
			$\sigma(i_r)=\pm j_r$ and $\sigma(j_r)=\pm i_r$ for all $1\le r\le t$,
			with ${\rm sign}(\sigma(i_r))={\rm sign}(\sigma(j_r))$; 
			\item
			if $\tau(n)=n$, namely $j_t<n$, then 
			$\sigma(n)=(-n)^{*d}$
			for some $d\in \ZZ$;
			and if $\tau(n) \ne n$, namely $j_t=n$, then 
			$\sigma(n)=i_t^{*d}$ and $\sigma(i_t)=n^{*-d}$ for some $d\in \ZZ$.
			%and $\sigma(i_t)=j_t^{*-d}=n^{*-d}$, and  ${\rm sign}(\sigma(i_r))={\rm sign}(\sigma(j_r))$ 
			%For $d\in \ZZ$ let ${\rm sign}(d):=1$ if $d> 0$ and $-1$ if $d\le 0$ (including $d=0$).\\
			
			%for all $1\le r< t$.
			%${\rm sign}(\sigma(i_t))={\rm sign}(\sigma(n))$. %=-{\rm sign}(d)$.
			%and $\sigma(n)=(-1)^{*d+1}n^{*+d}$ for some $d\in \ZZ$.
		\end{itemize}

		%\medskip
		
		For example, if $k=0$ then $T=\{id\}$, and the complete list of involutive $K_k$-coset representatives is
		\[
		R_0 = R_{0,id} 
		= \{[\pm 1, \pm 2, \dots, \pm (n-1), (-n)^{*d}] \mid d \in \ZZ\}.
		\]
		If $k=1$ then $T = \{id\} \sqcup \{(1,j) \mid 1 \le j \le n\}$,
		and the complete list of involutive $K_k$-coset representatives is the disjoint union 
		\[
		R_1 = \bigsqcup_{\tau\in T} R_{1,\tau}
		= R_{1,id} \sqcup \bigsqcup_{1 < j \le n} R_{1,(1,j)},
		\]
		where
		\[
		R_{1,id} := 
		\{[1, \pm 2, \dots, \pm (n-1), (-n)^{*d}] \mid d \in \ZZ\},
		\]
		\[
		R_{1,(1,j)} := 
		\{[\epsilon \cdot j, \pm 2, \dots, \pm (j-1), \epsilon \cdot 1, \pm (j+1),\dots,  \pm (n-1),(-n)^{*d}] \mid \epsilon \in \{-1,1\},\, d \in \ZZ \}\qquad (1 < j < n),
		\]
		and
		\[
		R_{1,(1,n)} :=
		\{[n^{*-d}, \pm 2, \dots, \pm (n-1), 1^{*d}]\mid d \in \ZZ \}.
		\]
		%for $\tau(n)=n$.
	\end{remark}
	%}

	\begin{remark}\label{rem:proto_HK}
		In contrast to Theorem~\ref{thm:main2_new},
		$(G,H_k)$ is not a proto-Gelfand pair. 
		In fact, 
		%letting $L_k:=\langle x_n\rangle H_k$, 
		$H_k$ is not even a proto-Gelfand subgroup of $\langle v \rangle \langle x_n \rangle  H_k$, 
		since the quotient $\langle v \rangle \langle x_n \rangle H_k / \langle x_n^3 \rangle H_k$ is isomorphic to $\S_3$, whose regular character is not multiplicity-free. 
		%as $H_kx_nH_k=x_nH_k\ne x_n^{-1}H_k=H_kx_n^{-1}H_k$.
	\end{remark}

	\section{Proto-Gelfand pairs for type $\tB_n$}\label{sec:tB}%\ \\ 
	
	Results for types $\tB_n$ follow quite easily from arguments similar to those used above for type $\tC_n$.
	
	%Recall the definition of the sign $\varepsilon(t)$ from Definition~\ref{def:sign}.
	
	%Recall Definition~\ref{def:exp}.
	%\begin{defn}\label{def:exp-sign}
	%Let $t$ be an integer. %which is not divisible by $2n+1$.
	%Write, uniquely, $t = k + (2n+1)\ell$, where $k, \ell \in \ZZ$ and $|k| \le n$. Define
	%\[
	%\lambda(t) := \ell
	%\]
	%and
	%\[
	%\lambda_{-}(t) = \max(0, -\lambda(t)).
	%\]
	%\end{defn}
	
	%\todo{YR: The following description of $\tB_n$, as the group of affine permutations
	%in $\tC_n$ with even sum of exponents, is different from but equivalent to the description in~\cite[Section~8.5]{BB}. 
	%In fact, it is more elegant, convenient and easy to apply.\\
	%It is not explicit in~\cite[Section~8.5]{BB}; a more precise reference or 
	%a short argument of the equivalence should be given.}
	
	Let $G$ be an affine Weyl group of type $\tC_n$, as in Subsection~\ref{sec:stabilizer}.
	Using Definition~\ref{def:exp},
	let $G_1$ consist of those $u = [u_1, \ldots, u_n] \in G$
	for which the sum of exponents in the window is even:
	\[
	\sum_{i=1}^{n} \lambda(u_i) \equiv 0 \pmod 2.
	\]
	$G_1$ is an index $2$ subgroup of $G$
	and, by~\cite[Section~8.5]{BB}, 
	it is an affine Weyl group of type $\tB_n$.
	
	\begin{remark}
		Formally, \cite[Section~8.5]{BB} considers the set of elements $u \in G$ satisfying
		\[
		\# \{j \le n \mid u_j \ge n+1\} \equiv 0 \pmod 2.
		\]
		However, it is not difficult to see that this description is equivalent to the above definition of $G_1$. 
		Indeed, for any $1 \le i \le n$,
		if $\lambda(u_i) = k  \ge 0$ then $j = i-(2n+1)b$ (for $b \in \ZZ$)
		satisfies both $j \le n$ and $u_j \ge n+1$ if and only if $0 \le b < k$, while $j = -i-(2n+1)b$ never satisfies both inequalities simultaneously. 
		If $\lambda(u_i) = -k < 0$ then $\lambda(u_{-i}) = k > 0$, leading to analogous conclusions.
		Overall we conclude that
		\[
		\# \{j \le n \mid u_j \ge n+1\} = \sum_{i=1}^{n} |\lambda(u_i)|,
		\]
		and this number is even if and only if 
		$\sum_{i=1}^{n} \lambda(u_i)$ is.
	\end{remark}
	
	Restrict the %homomorphism $\rho: G \to \Sym(X)$, 
	$G$-action $\rho$,
	from Definition~\ref{def:c-action_general_new}, to $G_1$.
	Using the notations of Section~\ref{sec:Gelfand_gen},
	for any $0 \le k \le n-1$,
	the stabilizer of $\omega_k = (0,\ldots,0,1,\ldots,1,0)$ under the restricted action is $H_k \cap G_1$. 
	Noting that $v := e_1e_2\cdots e_{n} = [-1,-2,\ldots,-n] \in G_1$,
	denote 
	$M_k := K_k \cap G_1 = \langle H_k,v \rangle \cap G_1$.
	The proof of the following result is very similar to that of Theorem~\ref{thm:main2_new}.
	%, almost verbatim, implies
	
	%\todo{RA: There are more orbits here than in the $G$-action, although probably not more stabilizers.
	%}
	
	\begin{theorem}\label{thm:main2_B}
		%The subgroup $M_k:=K_k\cap G_1$ is a proto-Gelfand subgroup of $G_1$.  
		$(G_1,M_k)$ is a proto-Gelfand pair for every $0\leq k\leq n-1$.
	\end{theorem}
	
	\begin{proof}
		By Proposition~\ref{t:H_k}, $H_k = H_k^L \times H_k^U$ where
		\[
		H_k^L = 
		\{[a_1^{*b_1}, \ldots, a_k^{*b_k}, k+1, \ldots, n] \mid
		(|a_1|, \ldots, |a_k|) \in S_k,\, 
		b_1, \ldots, b_k \in \ZZ\}
		\]
		and
		\[
		H_k^U = 
		\{[1, \ldots, k, a_{k+1}^{*b_{k+1}}, \ldots, a_n^{*b_n}] \mid
		(a_{k+1}-k, \ldots, a_n-k) \in S_{n-k},\, 
		b_{k+1}, \ldots, b_n \in \ZZ,\,
		b_{k+1} + \ldots + b_n = 0\}.
		\]
		Clearly $H_k^U \le G_1$ for any $k$, 
		and also $H_0^L = 1 \le G_1$. 
		Thus $H_0 = H_0^U \le G_1$ and $K_0 = \langle H_0,v \rangle \le G_1$. 
		Since $M_0 = K_0$ is a proto-Gelfand subgroup of $G$, it is also a proto-Gelfand subgroup of $G_1$.
		
		From now on assume that $1 \le k \le n-1$.
		$H_k^U \le G_1$ is an affine Weyl group of type $\tA_{n-k-1}$, while
		\[
		H_k^L \cap G_1 = 
		\{[a_1^{*b_1}, \ldots, a_k^{*b_k}, k+1, \ldots, n] \mid
		(|a_1|, \ldots, |a_k|) \in S_k,\, 
		b_1, \ldots, b_k \in \ZZ,\,
		b_1 + \ldots + b_k %\equiv 0 \pmod 2\}.
		\text{ is even}\}
		\]
		is an affine Weyl group of type $\tB_{k}$.
		%Proposition~\ref{t:H_k} implies that $H_k \leq G_1$ if and only if $k=0$. 
		%As $v\in G_1$ we have that $K_k\leq G_1$ if and only if $k=0$. As $M_0=K_0$ is already a proto-Gelfand subgroup of $G$, it is immediately a proto-Gelfand subgroup of $G_1$. We assume from now on that $k\geq 1$.
		%
		%The proof of Theorem~\ref{thm:main2_new} almost verbatim implies that $M_k=K_k\cap G_1$ is a proto-Gelfand subgroup of $G_1$. 
		By the Gelfand trick, it suffices to find an involution in each left $M_k$-coset in $G_1$. 
		%It follows from the proof of Proposition~\ref{t:H_k} that $H_k\cap G_1\cong B\times A$, where $A$ is an affine Weyl group of type $\tA_{n-k-1}$ and $B$ is an affine Weyl group of type $\tB_{k}$.
		
		Using the natural embeddings $\tB_k \le \tC_k$ and $\tA_{n-k-1} \le \tB_{n-k} \le \tC_{n-k}$, the parity condition of type $\tB$ implies
		\[
		\begin{aligned}
		\left( \tC_k \times \tC_{n-k} \right) \cap \tB_n 
		&= \left( \tB_k \times \tB_{n-k} \right) \sqcup \left( x_1\tB_k \times x_n\tB_{n-k} \right) \\
		&= \left( V_L\tB_k \times \langle x_n^2 \rangle V_U\tA_{n-k-1} \right) \sqcup \left( x_1 V_L\tB_k \times x_n \langle x_n^2 \rangle V_U\tA_{n-k-1} \right) \\
		&= \{1, x_1 x_n\} \langle x_n^2 \rangle V \left( \tB_k \times \tA_{n-k-1} \right).
		\end{aligned}
		\]
		By Observation~\ref{obs:transversal}, $\tC_n = T \left( \tC_k \times \tC_{n-k} \right)$ where $T \subseteq G_1$ is the set of representatives from Definition~\ref{defn:T}.
		Hence
		\[
		\begin{aligned}
		G_1 = \tB_n
		&= T \left[ \left( \tC_k \times \tC_{n-k} \right) \cap \tB_n \right]
		= T \{1, x_1 x_n\} \langle x_n^2 \rangle V \left( \tB_k \times \tA_{n-k-1} \right) \\
		&= T \{1, x_1 x_n\} \langle x_n^2 \rangle V \left( H_k \cap G_1 \right)
		%= T \{1, x_1 x_n\} \langle x_n^2 \rangle V \langle v \rangle \left( H_k \cap G_1 \right) \\
		= T \{1, x_1 x_n\} \langle x_n^2 \rangle V M_k.
		\end{aligned}
		\]
		Thus every left coset of $M_k$ in $G_1$ is of the form
		$\tau x_n^{2d} w M_k$ or $\tau x_1x_n^{2d+1} w M_k$ for some $\tau \in T$, $w \in V$ and integer $d$.
		Note that $V \cap H_k = V_L \le G_1$ and thus $V \cap M_k = V_L \langle v \rangle$.
		In the expression of $\tau$ as a product of disjoint cycles, each $2$-cycle contains one element of $[k]$ and one element of $[n] \setminus [k]$. It follows that there exists an element $w' \in V \cap M_k$ such that
		$ww'$ commutes with $\tau$ and $x_n^{ww'} = x_n^{-1}$.
		
		If $\tau(n) \ne n$ then $\tau(n)\le k$, hence
		$x_{\tau(n)}^{-2d}, (x_1^{-1})^{ww'} x_{\tau(n)}^{-(2d+1)} \in M_k$. 
		Thus the coset representatives 
		\[
		\tau x_n^{2d} ww' x_{\tau(n)}^{-2d} 
		\in \tau x_n^{2d} w M_k
		\]
		and 
		\[
		\tau x_n^{2d+1} ww' x_{\tau(n)}^{-(2d+1)} 
		%= \tau x_1x_n^{2d+1} x_1^{-1} ww' x_{\tau(n)}^{-(2d+1)}
		= \tau x_1x_n^{2d+1}  ww' (x_1^{-1})^{ww'} x_{\tau(n)}^{-(2d+1)}
		\in  \tau x_1x_n^{2d+1} w M_k
		\] 
		are involutions:
		\[
		\left( \tau  x_n^{2d} ww' x_{\tau(n)}^{-2d} \right)^2
		= x_{\tau(n)}^{2d} ww' x_n^{-2d}x_n^{2d}ww' x_{\tau(n)}^{-2d}
		= 1
		\]
		and similarly
		\[
		\left( \tau x_n^{2d+1} ww' x_{\tau(n)}^{-(2d+1)} \right)^2
		= x_{\tau(n)}^{2d+1} ww' x_n^{-(2d+1)}x_n^{2d+1} ww' x_{\tau(n)}^{-(2d+1)}
		= 1.
		\]
		
		Now, assume that $\tau(n) = n$. Then
		\[
		\tau x_n^{2d} ww' \in \tau  x_n^{2d} w M_k
		\] 
		is an involution:
		\[
		(\tau x_n^{2d} ww')^2
		= x_n^{2d} ww' x_n^{2d} ww'
		= x_n^{2d}(x_n^{2d})^{ww'}
		= 1.
		\]
		If, in addition, $\tau(1) \ne 1$ 
		then $x_{\tau(1)}^{-1}x_n \in M_k$ so that
		\[
		\tau x_1x_n^{2d+1} ww' x_{\tau(1)}^{-1}x_n 
		\in \tau  x_1x_n^{2d+1} w M_k
		\]
		is an involution:
		\[
		\begin{aligned}
		(\tau x_1x_n^{2d+1} ww' x_{\tau(1)}^{-1}x_n)^2
		&= x_{\tau(1)}x_n^{2d+1} ww' x_1^{-1}x_nx_1x_n^{2d+1} ww' x_{\tau(1)}^{-1}x_n \\
		&= x_{\tau(1)}x_n^{2d+1} (x_n^{2d+2})^{ww'} x_{\tau(1)}^{-1}x_n
		= 1.
		\end{aligned}
		\]
		
		Finally, assume that $\tau(n)=n$ and $\tau(1)=1$. 
		%Then $\tau$ commutes with $x_1$ and $x_n$.
		Recall that $w' \in V \cap M_k = V_L \langle v \rangle$
		was chosen so that $ww'$ commutes with $\tau$ and $x_n^{ww'} = x_n^{-1}$. In our case, we can redefine $w'(1) \in \{1,-1\}$, if necessary, so that these properties are preserved but also $x_1^{ww'}=x_1^{-1}$. Then
		\[
		\tau x_1x_n^{2d+1} ww' 
		\in \tau  x_1x_n^{2d+1} w M_k
		\]
		is an involution:
		\[
		(\tau x_1x_n^{2d+1} ww')^2
		= x_1x_n^{2d+1} (x_1x_n^{2d+1})^{ww'} =1,
		\] 
		and this completes the proof.

	\end{proof}
	
	%\todo{YR: Using window notation it is not difficult to find a
	\begin{remark}
		One can apply the window notation to get an explicit description of a %complete list
		complete list of involutive $M_k$-coset representatives in $\tB_n$. It is a disjoint union %of subsets indexed by $\tau\in T$ 
		\[
		R := \bigsqcup_{\tau\in T} R_{k,\tau}.
		\]
		%\todo{RA:
		%Apparently, they should be also indexed by $w$ and $m$.
		%}
		
		%, where the affine permutations in these subsets are defined by the following window:
		%\[
		%\{[\pm 1, \pm 2, \dots, \pm (n-1),((-1)^{*d+1}n)^{*+d}]\mid d\in \ZZ\},
		%\]
		%(for $t=0$);
		For every $\tau = (i_1,j_1) \cdots (i_t,j_t) \in T$, let 
		%associate 
		$R_{k,\tau}$ be the set of affine permutations $\sigma=[\sigma(1),\ldots,\sigma(n)]$, in window notation, satisfying:
		\begin{itemize}
			\item
			$\sigma(i) = i$ for all $i\in [k]\setminus \{1,i_1,\dots,i_t\}$;
			%\item if $\tau(1)=1$ then $\sigma(1)=1$.
			\item
			$\sigma(i) = \pm i$ for all  
			$k < i\not\in \{j_1,\dots,j_t,n\}$;
			\item 
			$\sigma(i_r)=\pm j_r$ and $\sigma(j_r)=\pm i_r$ for all $1\le r\le t$,
			with ${\rm sign}(\sigma(i_r))={\rm sign}(\sigma(j_r))$;
			\item
			if $\tau(n) \ne n$, namely $j_t=n$, then 
			$\sigma(n)=i_t^{*d}$ and $\sigma(i_t)=n^{*-d}$ for some (even or odd) $d \in \ZZ$
			and if in addition $\tau(1)=1$ then $\sigma(1)=1$.
			\item
			%For $d\in \ZZ$ let ${\rm sign}(d):=1$ if $d> 0$ and $-1$ if $d\le 0$ (including $d=0$).\\
			if $\tau(n)=n$, namely
			$j_t\ne n$, then there are two families of associated representatives.
			The first is defined by 
			%$\sigma(n)=-{\rm sign}(d)\cdot n^{*d}$ 
			$\sigma(n)=(-n)^{*d}$
			for some even $d\in \ZZ$, and if in addition $\tau(1)=1$ then $\sigma(1)=1$.
			%$\sigma(i_r)=\pm j_r$ and $\sigma(j_r)=\pm i_r$
			%with ${\rm sign}(\sigma(i_r))={\rm sign}(\sigma(j_r))$ for every $1\le r\le t$;
			%and $\sigma(i)=\pm i$ for all $ i\not\in \{i_1,\dots,i_t,j_1,\dots,j_t,n\}$.
			
			%\medskip
			
			%$R_{\tau,1}$ is defined similarly. %, where even $d$ is replaced by odd $d$.
			
			%If $\tau(n) \ne n$ then the extra condition is 
			%$\sigma(n)=i_t^{*+d}$ and $\sigma(i_t)=n^{*-d}$ for some odd $d\in \ZZ$.
			The second family splits into two subcases: 
			If $\tau(n)=n$  and $\tau(1)=j\ne 1$ then %there is no $i\le k$  with $\tau(i)=i$, 
			%let $i_0$ be the minimal $i\le k$
			%letting $j:=\tau(1)\ne n$ 
			%let 
			$\sigma(1)=j^{*+1}$,\ $\sigma(j)=1^{*+1}$ 
			and $\sigma(n)=(-n)^{*d}$ for some even $d\in \ZZ$.
			
			Finally, if $\tau(n)=n$  and  $\tau(1)=1$
			%there exists  $i\le k$ with $\tau(i)=i$ then let $i_0:=\min_i \{\tau(i)=i\}$; 
			we let  
			$\sigma(1)=(-1)^{*+1}$ and $\sigma(n)=(-n)^{*d}$ for some odd $d\in \ZZ$. 
			%}
		\end{itemize}
	\end{remark}

	\section{Combinatorial flip actions - revisited}\label{sec:combin-revisited}
	
	In this section we apply Theorem~\ref{thm:main2_new} 
	%Corollary~\ref{cor:main}
	to prove %a generalized version of  
	Theorem~\ref{thm:main1} and its analogues (Propositions~\ref{prop:CTFT} and~\ref{prop:Cn-words}).
	
	\smallskip
	
	Recall the action of $G \cong \tC_n$ on $X = \ZZ_3^n \times \ZZ$ defined by the map $\rho$ from Definition~\ref{def:c-action_general_new}.
	For each $0 \le k \le n-1$ and a positive integer $m$ let 
	\begin{equation}\label{eq:omega_m_def}
	\Omega_{n,k,m} := \Omega_{n,k}/\langle (0,\ldots,0,m) \rangle
	= \{(a_1,\dots,a_n,b)\in \ZZ_3^n\times \ZZ_m\mid k=|\{i \mid a_i = 0\}|\},
	%\Omega_{n,k}/m\ZZ=\{(a_1,\dots,a_n,b):\ \forall i\ a_i\in \ZZ_3,\ |\{i:\ a_i\ne 0\}|=k\ ,  b\in \ZZ_m\}.
	\end{equation}
	%Consider the quotient
	and let $\Omega_{n,k,m}^\pm$
	be the quotient set of $\Omega_{n,k,m}$ under the equivalence relation 
	\[
	(a_1,\ldots,a_n,b) \sim (-a_1,\ldots,-a_n,-b).
	\]
	
	%\smallskip
	
	\begin{observation}\label{obs:main}
		The map $\rho$ determines a well-defined $\tC_n$-action on $\Omega_{n,k,m}^\pm$.
	\end{observation}
	
	\begin{corollary}\label{cor:main}
		The $\tC_n$-action on $\Omega_{n,k,m}^\pm$,
		determined by $\rho$,
		is multiplicity-free.
	\end{corollary}
	
	\begin{proof}
		Let $K_{k,m}$ be the stabilizer of $\pm\omega_k \in \Omega_{n,k,m}^\pm$.
		Since $K_k = \langle H_k, v \rangle \le K_{k,m}$,  Theorem~\ref{thm:main2_new} implies that the pair $(\tC_n, K_{k,m})$ is proto-Gelfand. As $\Omega_{n,k,m}^\pm$ is finite, 
		Corollary~\ref{cor:free_action} completes the proof.
		
	\end{proof}

	\medskip
	
	In the rest of this section we take $m=n+r$, 
	where $r=2$, $4$ and $3$ in Subsections~\ref{sec:partial_arc},
	%$k=4$ in Subsection
	\ref{sec:combin-revisited.triangulations} and
	\ref{sec:combin-revisited.factorizations}, respectively. 
	Addition is always modulo $m$.

	\begin{figure}[h]\caption{The domains under the action of $\tC_{n}$ and their correspondence to the general action on $\Omega_{n,k}$.}\label{fig:overview}
		\begin{tabular}{l|l|l}
			Domain of action of $\tC_{n}$ &Name &$\Omega$ notation
			%&equivariant involution
			\\
			\hline
			General action& $\ZZ_3^n\times\ZZ_m$&$\Omega_{n,k,m}$
			%&$\omega\mapsto-\omega$
			\\
			Arc permutations&$\A_{n+2}$&$\Omega_{n,0,n+2}$%&reversal, fixing $n+2$
			\\
			Partial arc permutations&$\A_{n+2,k}$&$\Omega_{n,n-k,n+2}$
			%&reversal, fixing $n+2$
			\\
			Colored triangle-free triangulations&$CTFT(n+4)$&$\Omega_{n,0,n+4}$
			%&reflection through a diagonal
			\\
			Factorizations of the Coxeter element&$LF_{n+3}$&$\Omega_{n,0,n+3}$%&????
			\\
			Geometric caterpillars&$GC_{n+3}$&$\Omega_{n,0,n+3}$%&????
			\\
		\end{tabular}
	\end{figure} 
	
	\subsection{%A %$\tC_{n}$-
		%flip action on a
		Partial arc permutations}\label{sec:partial_arc}%\ \\
	
	In this section we generalize and prove Theorem~\ref{thm:main1}.
	%We begin with a typical example.
	%For a set $X$ let $\Sym(X)$ be the symmetric group of 
	%%bijections
	%permutations from $X$ to itself.
	%Let $\S_n:=\Sym([n])$ 
	%be the symmetric group on the set %letters 
	%$[n]:=\{1,\dots,n\}$.
	We begin with a generalization of the flip action  described in Subsection~\ref{sec:arc}, where % to partial permutations.
	%In this case, 
	the $\tC_n$-action on $\ZZ_3^n\times \ZZ_m$, $m=n+2$, is interpreted as a natural action on partial arc permutations.
	
	\medskip
	
	The set of {\em partial permutations}  $\S_{n,k}$ 
	consists of all one-to-one mappings from 
	%subsets of $[n]:=\{1,2,\ldots,n\}$ of size $k$ to subsets of $[n]$.
	$\{J\subseteq [n]:\ |J|=k\}$ to itself.
	A partial permutation may be represented by a sequence of $n$ symbols
	$\pi=[\pi(1),\dots,\pi(n)]$, some of which are distinct elements in $[n]$ and the remaining are denoted by 
	$\circ$. For example, $\pi=[5,\circ ,2,6,\circ,\circ,3]$ is a partial permutation in $\S_{7,4}$.
	
	A transposition $\sigma_i:=(i,i+1)$ acts on $\pi$ (on the right) by switching the $i$-th and $(i+1)$-st entries.
	For example, for $\pi$ as above, $\pi\sigma_4=[5,\circ,2,\circ,6,\circ,3]$.
	
	\medskip
	
	An {\em interval} in the cyclic group $\ZZ_m$ is a subset of the form
	%$[i,i+k]:=
	$\{i, i+1, \ldots, i+d\}$, where addition is modulo $m$.
	
	\begin{defn}\label{def:arcpp}
		For every $0 < k \le m-2$, define the set of {\em partial arc permutations} $\A_{m,k}$
		as the set of partial permutations 
		$\pi \in \S_{m,k+2}$  %$\cup \S_{n,k-1}$, 
		satisfying the following conditions:
		\begin{itemize}
			\item[(i)]
			Every suffix of $\pi$ forms an interval in $\ZZ_{m}$, where the letters $\circ$ are ignored;
			
			\item[(ii)]  
			$|\{1<i< m\mid \pi(i)\ne \circ\}|=k$;
			
			\item[(iii)] 
			$\pi(i)\ne \circ$ for $i=1,m$;
			
			\item[(iv)] 
			If $i_0:=\min\{ 1<i\mid \pi(i)\ne \circ\}$ then
			\[
			\pi(1) = \begin{cases} 
			%\circ ,&\text{if } \pi(2)=\circ;\\
			\pi(i_0)-k-1, &\text{if } \pi(i_0)-1\in\{\pi(i_0+1),\dots,\pi(m)\};\\
			\pi(i_0)+k+1,&\text{if } \pi(i_0)+1\in\{\pi(i_0+1),\dots,\pi(m)\}.
			\end{cases}
			\]
			%Either $\pi(1)=\pi(2)=\circ$, or $\pi(1)=\pi(2)+k+1\mod n$ if there exists $2\le j\le n$ such that $\pi(j)=\pi(2)-1\mod n$,
			%or  $(\pi(1)=\pi(2 -(k+1)\mod n$ otherwise.
			%\todo{RA: Why use this unusual definition of $\pi(1)$?}
		\end{itemize}
	\end{defn}
	
	%the set of partial arc permutations in $\S_{n,k}$.
	%\begin{example} T
	For example, the partial permutation $[7,4,\circ,\circ,8,\circ,3,1,2]$ is an arc permutation in $\A_{8,4}$.
	%\end{example}
	
	\medskip
	
	By conditions (i)-(iv) in Definition~\ref{def:arcpp}, 
	%are motivated by the following.
	%specialization at $k=n$ generalized $\tC_n$ flip action.
	$\A_{n+2,n}=\A_{n+2}$
	is just the set of arc permutations in $\S_{n+2}$. %\subset \S_{n+2,n+2}=\S_{n+2}$.
	%\medskip
	%We now generalize the 
	Furthermore, %these conditions provide a natural generalization of 
	the $\tC_n$-action on %arc permutations 
	$\A_{n+2}$, described in Subsection~\ref{sec:arc}, 
	%%introduced in~\cite{TFT2},%and describe a natural 
	may be naturally generalized 
	to a flip action %of the group $\tC_{n}$ 
	on the set of partial arc permutations
	$\A_{n+2,k}$, for every $0<k\le n$. %Denote the adjacent transposition $(i,i+1)$  by $\sigma_i$.
	
	\begin{defn}\label{def:c-action_partial_arc}
		For every $0\le i\le n$ and $0< k\le n$, define %a homomorphism $\rho: \tC_{n} \to \Sym(\A_{n+2})$ 
		%as a multiplicative extension of the following map: by
		\[
		\rho^A(s_i)(\pi) := \begin{cases}
		\pi \sigma_{i+1},&\hbox{\rm if }\pi \sigma_{i+1}\in \A_{n+2,k};\\
		\pi, &\hbox{\rm otherwise;}
		\end{cases} \qquad (\forall \pi\in \A_{n+2,k}).
		\]
	\end{defn}
	
	%We will now show 
	%%in Section~\ref{sec:combin-revisited}
	%%was shown in~\cite{TFT2} 
	%that 
	%%for every $0\le k\le n+2$
	%%the above map 
	%this map determines a well defined %transitive 
	%$\tC_{n}$-action on $\A_{n+2,k}$, and that furthermore,
	%modulo a natural involution, this is a multiplicity-free
	%$\tC_n$ permutation representation.
	
	Recall the definition of 
	$\Omega_{n,k,m}$ from equation~\eqref{eq:omega_m_def}.
	
	\begin{proposition}\label{prop:action_partial_arc}
		%\begin{itemize}
		%\item[1.] 
		For every $0< k\le n$, the map $\rho^A$ determines a well defined %transitive 
		$\tC_{n}$-action on $\A_{n+2,k}$, which is isomorphic to its action on
		$\Omega_{n,n-k,n+2}$.
		%\item[2.] 
		%\end{itemize}
	\end{proposition}
	
	\begin{proof}
		%Denote 
		%\[
		%\Omega_{n,k,m}:=\Omega_{n,k}/\langle (0,\dots,0,m) \rangle
		%\]
		%That is
		%\[
		%\Omega_{n,k,m}=\{(a_1,\dots,a_n,b)\in \ZZ_3^n\times \ZZ_m \mid k=|\{i\mid a_i\ne 0\}|\}.
		%\]
		For every $\pi\in \A_{n+2, k}$, define
		$\phi(\pi)=(a_1,\dots,a_n,b)\in \Omega_{n,n-k,n+2}$ by 
		letting %,  for all $1\le i\le n$
		%, for $0\le i<n$ 
		\[
		a_i:=
		\begin{cases} 
		0, & \pi(i+1)=\circ;\\
		1,& \pi(i+1)-1\in \{\pi(i+2),\dots,\pi(n+2)\};\\
		-1,& \pi(i+1)+1\in \{\pi(i+2),\dots,\pi(n+2)\},
		\end{cases} \qquad (\forall\ 1\le i\le n),
		\]
		and $b:=\pi(n+2)$.
		
		\smallskip
		
		%\noindent
		%\todo{
		For example,
		for $\pi=[8,\circ,5,1,\circ,4,2,3]\in \A_{8,4}\subset \S_{8,6}$,\ \  
		%Then
		$\phi(\pi)=(0,1,-1,0,1,-1,3)\in \Omega_{6,2,8}$. 
		%\subset \ZZ_3^6\times \ZZ_8$. \ 
		%Letting $\pi=[1,5,\circ,\circ,\circ,4,2,3]\in \A_{8,3}$,
		%$\phi(\pi)=(1,0,0,0,1,-1,3)\in \Omega_{6,3,8}$, %$\ZZ_3^6\times \ZZ_8$,
		
		%$\bar v:=(-1,1,0,0,0,1,-1,3)$,
		%$\phi(\bar v)$ be the partial arc permutation
		%\\
		%Letting $v=(0,1,0,1,-1,0,3)\in \ZZ_3^6\times \ZZ_8$,
		%$\bar v:=(0,0,1,0,1,-1,0,3)$.\\
		%Let $\phi(\bar v)$ be the partial arc permutation
		%\[
		%\phi(\bar v)=[0,0,5,0,4,2,0,3]\in \A_{8,4}.
		%\]
		%}
		
		\medskip
		
		To prove that $\phi$ is a bijection, it suffices to show that it is invertible. Indeed, for every vector $\bar a=(a_1,\dots,a_n,b)\in \Omega_{n,n-k,n+2}$
		let %$\bar v=(a_0,a_1,\dots,a_n,b)$ with 
		$i_0:=\min\{i>1\mid a_i\ne 0\}$ and
		$a_0:=-a_{i_0}$.
		Let $\phi^{-1}(\bar a)$ be the partial arc permutation $\pi\in \A_{n+2,k}$ defined by
		\[
		\pi(i):=\begin{cases} 
		b ,& i=n+2;\\
		\circ, & a_{i-1}=0;\\
		b+ \sum\limits_{i-1\le j\le n:\ a_j=1} a_j,& a_{i-1}=1;\\
		b+ \sum\limits_{i-1\le j\le n:\ a_j=-1} a_j,& a_{i-1}=-1.
		\end{cases}
		\]
		One can easily verify that this is indeed %determines 
		the inverse map.
		
		\medskip
		
		Finally, observe that for every $0\le i\le n$ and $\pi\in \A_{n+2,k}$
		\[
		\phi(\rho^A(s_i)(\pi))= \rho(s_i)(\phi(\pi)).
		\]
		Observation~\ref{obs:relations} completes the proof.
		
	\end{proof}
	
	%\bigskip
	
	%Let $c:=\sigma_1\sigma_2\cdots\sigma_{n-1}=(1,2,\dots,n)$
	%a Coxeter element in $\S_n$.
	
	For $\pi\in \A_{n+2,k}$,  define
	$\pi^\iota\in  \A_{n+2,k}$ by
	%$\pi^\iota:= c^{\pi(n)} w_0 \pi c^{-\pi(n)}$ %= [n+2-\pi(1),\dots,n+3-\pi(1)]$,
	%where $w_0=[n+1,\ldots,1]$ denotes the longest element in $\S_{n+1}$.
	%In other words, for every $1\le i\le n$ let
	\[
	\pi^{\iota}(i):=\begin{cases} 
	%2\pi(n)-\pi(i),
	
	\pi(i),&\text{if }\pi(i)=\circ\text{ or }\pi(i)=n+2;\\
	n+2-\pi(i),&\text{otherwise};%\pi(i)\ne \circ;\\
	%\pi(i) ,& \pi(i)=\circ \ \text{or}\ i=n;\\  % \text{otherwise}.
	\end{cases} \qquad (1\le i\le n+2).
	\]
	The set of partial arc permutations $\A_{n+2,k}$ is
	closed under $\iota$
	%left multiplication by $w_0$, 
	and this operation commutes with the above $\tC_n$-action, namely, 
	%\todo{To be checked}
	%namely,
	%%furthermore, 
	%for every %$1\le i\le n$ and 
	%$\pi\in \A_{n+2}$
	\[
	\rho^A(s_i) (\pi^\iota) = (\rho^A(s_i)(\pi))^\iota  
	\qquad (\forall \pi\in \A_{n+2,k},\, 0 \le i \le n).
	\]
	Hence, the 
	$\tC_{n}$-action on $\A_{n+2,k}$ determines 
	a well defined $\tC_{n}$-action on the set of equivalence classes $\A_{n+2,k}/\iota$.
	
	%\todo{YR: Is this indeed the effect of multiplication by $v$?}
	
	\begin{theorem}\label{thm:main11}
		For every $0<k\le n$,
		%$0\le k< n-1$,
		the $\tC_{n}$-module determined by the above action on $\iota$-equivalence classes of the set of partial arc permutations $\A_{n+2,k}$
		%$\A_{n+2}$ modulo the reverse operation is
		%the permutations 
		is multiplicity-free.
	\end{theorem}
	
	\begin{proof}
		%First, observe that
		%\[
		%\phi^{-1}(v\phi(\pi))=\pi^\iota \qquad (\forall\ \pi\in \A_{n,n+2}),
		%\]
		Observe that
		\[
		\phi(\pi^\iota)=-\phi(\pi) \qquad (\forall \pi\in \A_{n+2,k}),
		\]
		where $\phi:\A_{n+2,k}\longrightarrow \Omega_{n,k,n+2}$ 
		is the $\tC_n$-module isomorphism 
		from the proof of Proposition~\ref{prop:action_partial_arc}.
		It follows that
		\[
		\A_{n,k}/\iota\cong \Omega^\pm_{n,k,n+2}, %/v,
		\]
		as $\tC_n$-modules. Corollary~\ref{cor:main} completes the proof.
		%\todo{YR: Allign this proof with current Section 3}
		
	\end{proof}
	
	\begin{remark}
		Letting $k=n$ in Theorem~\ref{thm:main11} implies Theorem~\ref{thm:main1}.
	\end{remark}
	
	%\newpage
	
	%\subsection{Other combinatorial applications}
	
	\bigskip
	
	\subsection{Triangulations}\label{sec:combin-revisited.triangulations} 
	
	\begin{proof}[Proof of Proposition~\ref{prop:CTFT}]
		Recall from Subsection~\ref{ss:triangles} that
		a colored triangle-free triangulation $T\in CTFT(n+4)$ is encoded by a sequence of diagonals $(d_0,\dots,d_{n})$ which ends with a short diagonal $d_{n}=(b-1,b+1)$ for some
		$0\le b<n+4$. 
		This sequence may be encoded, in turn, by a vector in $\phi(T)\in \Omega_{n,0,n+4}$
		as follows.  
		%that the sequence of diagonals $(d_1,\dots,d_{n+1})$ ends with a short diagonal %$d_{n+1}=(b-1,b+1)$, for some 
		First, let 
		%$b\in \ZZ_{n+4}$ and let
		\[
		\phi(T)(n):= b,
		\]
		if  $d_{n}=(b-1,b+1)$. 
		
		Next notice that for every $1\le i\le n$ the vertices in $d_{i+1},d_{i+2},\dots,d_n$
		together with $b$ form an interval in $\ZZ_{n+4}$.
		Denote this interval by $[a_{i+1},c_{i+1}]$, see Figure~\ref{fig:11}.
		For every $1\le i\le n$ let
		\[
		\phi(T)(i+1):= \begin{cases} 
		1,& d_{i}= (a_{i+1},c_{i+1}+1);\\
		-1,& d_{i}= (a_{i+1}-1,c_{i+1}).
		\end{cases}
		\]

		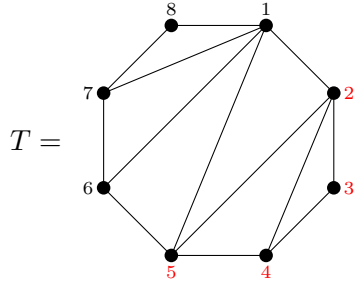
\begin{figure}[htb]
			\begin{center}
				\begin{tikzpicture}[scale=0.9]
				\draw (-1,1.7) node {$T=$};
				\fill (2.4,3.4) circle (0.1) node[above]{\tiny 1}; \fill (3.4,2.4)
				circle (0.1) node[right][red]{\tiny 2}; \fill (3.4,1) circle (0.1)
				node[right][red]{\tiny 3}; \fill (2.4,0) circle (0.1) node[below][red]{\tiny
					4}; \fill (1,0) circle (0.1) node[below][red]{\tiny 5}; \fill (0,1)
				circle (0.1) node[left]{\tiny 6}; \fill (0,2.4) circle (0.1)
				node[left]{\tiny 7}; \fill (1,3.4) circle (0.1) node[above]{\tiny 
					8};
				
				\draw
				(1,3.4)--(2.4,3.4)--(3.4,2.4)--(3.4,1)--(2.4,0)--(1,0)--(0,1)--(0,2.4)--(1,3.4);
				\draw(3.4,2.4)--(1,0);\draw (0,1)--(2.4,3.4); \draw
				(2.4,3.4)--(1,0);\draw (3.4,2.4)--(2.4,0); \draw
				(0,2.4)--(2.4,3.4);
				
				\end{tikzpicture}
			\end{center}
			\caption{$T\in CTFT(8)$ with diagonal sequence  $\left((1,7), (1,6), (1,5), (2,5), (2,4)\right)$,
				encoded by $\phi(T)=(1,1,-1,1,3)$. The interval $[a_3,c_3]=\{2,3,4,5\}$ consists of $b=3$ and the labels of the vertices the diagonals $d_3$ and $d_4$, hence $\phi(T)(3)=-1$.} 
			\label{fig:11}
		\end{figure}
		
		Observing that, for every $0\le i\le n$ and $T\in CTFT(n+4)$
		\[
		\rho(s_i)(\phi(T))=\phi(s_i T),
		\]
		$CTFT(n+4)$ and $\Omega_{n,0,n+4}$ are isomorphic $\tC_n$-modules. 
		
		Also
		\[
		\iota(T)%^\iota
		= \phi^{-1}(-\phi(T)) \qquad (\forall T\in CTFT(n+4)).
		\]
		It follows that
		\[
		CTFT(n+4)/\iota \cong \Omega^\pm_{n,0,n+4}, %/v,
		\]
		as $\tC_n$-modules. 
		Corollary~\ref{cor:main} completes the proof.
		%\todo{YR: Missing at the moment: (1) take in account the involution $\iota$.\\
		%(2) The following statement:
		%\[
		%CTFT(n+4)/\iota\cong \Omega^\pm_{n,0,n+4}, %/v,
		%\]
		%as $\tC_n$-modules.}
		
	\end{proof}
	
	\subsection{Factorizations}\label{sec:combin-revisited.factorizations} 
	
	%Proposition~\ref{prop:Cn-words} is proved in this subsection.
	
	\begin{proof}[Proof of Proposition~\ref{prop:Cn-words}]
		Define a bijection $\phi: LF_{n+3}\rightarrow \Omega_{n,0,n+3}$ as follows.
		First, recall from~\cite[Prop. 4.5]{YK} that 
		for every factorization $w=(t_1,\dots,t_{n+2})\in LF_{n+3}$,
		%$t_{n+2}=(j, j+1)$ for some $1\le j\le n$, and let
		$t_1 = (j,j+1)$ for some $1\le j\le n+3$, and let
		\[
		\phi(w)(n+1):= j. % \mod n.
		\]
		For every $1 \le i\le n$ let
		\[
		\phi(w)(n+1-i):=\begin{cases} 
		1 ,& t_{i+1} = (j,j+1)\ \text{for some}\ 1\le j\le n+3;\\
		-1,& \text{otherwise}.
		\end{cases}
		\]
		To verify that $\phi$ is a bijection we define an inverse map. 
		First, let $t_1:=(\phi(w)(n+1), \phi(w)(n+1)+1)$.
		For every $1 \le i \le n$, define $t_{i+1}$ by induction.
		%Observe 
		Recall that for every $1\le i\le n+2$ 
		the product  $t_{1}t_{2}\cdots t_{i}$ is equal to a cycle $(a_{i},a_{i}+1,\dots,c_{i})$, where 
		$[a_{i},c_{i}]$ forms an interval in $\ZZ_{n+3}$ of length $i+1$~\cite[Lemma 4.8]{YK}.
		%Observe 
		It follows 
		that $t_{i+1}$ is equal  to either $(c_{i}, c_i+1)$ or to 
		$(a_{i}-1,c_{i})$. %[reference ??]. 
		Let $t_{i+1}:=(c_{i},c_{i}+1)$
		if $\phi(w)(n+1-i)=1$
		and $t_{i+1}:=(a_{i}-1,c_{i})$ if $\phi(w)(n+1-i)=-1$.
		Finally, let $t_{n+2}:=
		\gamma t_{n+1}t_{n}\cdots t_1=\gamma (t_1t_2\cdots t_{n+1})^{-1}$, where $\gamma:=(1,2,\dots,n+3)$. By~\cite[Lemma 4.8]{YK}, 
		$t_1 t_2\cdots t_{n+1}=(a_{n+1}, a_{n+1}+1,\dots, a_{n+1}+n+1)$; 
		it follows that $t_{n+2}$ is also a transposition (in fact, an adjacent transposition). 
		
		%\todo{PH: I would say that $t_{n+2}$ is also a transposition.
		%in fact it is $(1,2)$, $(1,n)$ or $(n-1,n)$. Am I right?\\
		%YR: It follows from~\cite[Prop. 4.5]{YK} that $t_{n+1}$ is an adjacent transposition, not necessarily of the above form.}
		
		%Claim: $\phi$ is a bijection
		
		Furthermore, the map $\rho^{LF}$ from Subsection~\ref{ss:words} determines a well-defined 
		$\tC_n$-action, which is isomorphic via $\phi$ to the $\tC_n$-action on $\Omega_{n,0,n+3}$, since
		\[
		\phi(\rho^{LF}(s_i)(w)) 
		= \rho(s_{n-i})\phi(w)
		\qquad (\forall w\in LF_{n+3},\ 0\le i\le n).
		\] 
		Letting $\iota: LF_{n+3}\rightarrow LF_{n+3}$ be defined as
		%$\iota:=\phi^{-1}v\phi$, 
		\[
		w^\iota:=\phi^{-1}(-\phi(w))\qquad (\forall w\in LF_{n+3}),
		\]
		we deduce that $LF_{n+3}/\iota\cong \Omega^\pm_{n,0,n+3}$. By Corollary~\ref{cor:main}, it is multiplicity-free. 
		
		%}
		%\todo{YR: Missing, more explicit study of the involution $\iota$.}

	\end{proof}
	
	%\subsection{Geometric caterpillars}\label{sec:combin-revisited.caterpillars} 

	%\todo{Proofs in this section should be cleaned/clarified}

	%%%%%%%%%%%%%%%%%%%%%%%%%%%%%%%%%%%%%%%%%%%%%%%
	\section{Final remarks and open problems}\label{sec:final}
	
	\subsection{The action on $\Omega_{n,k,m}$}
	%finite case}
	%{Appendix}
	In Corollary~\ref{cor:main} we proved that the action of $\tC_n$ on each orbit-quotient $\Omega_{n,k,m}^\pm=\Omega_{n,k,m}/\sim$ is multiplicity-free. 
	Here we treat the action on $\Omega_{n,k,m}$. %A stronger result holds when $m\leq 2$.
	
	%Here we prove the following.
	
	\begin{proposition}\label{prop:m=2}
		%that for 
		Let $G$ be an affine Weyl group of type $\tC_{n}$ or $\tB_n$. The action of $G$ on the orbit $\Omega_{n,k,m}$ is multiplicity-free if and only if $m\leq 2$. In particular, none of 
		the original finite combinatorial actions on 
		arc permutations, geometric caterpillars and triangle-free triangulations is multiplicity-free. 
	\end{proposition}
	
	\begin{proof}
		The proof here is for the case when $G$ is of type $\tC_n$. The other case can be treated similarly, see the proof of Theorem~\ref{thm:main2_B}.
		
		We use the notations from Section~\ref{sec:Gelfand_gen}.
		
		The group $G$ acts on the finite set $\ZZ_3^n\times\ZZ_m$ via $\rho$ with orbits $\Omega_{n,k,m}$, see \eqref{eq:omega_m_def}. As in Corollary~\ref{cor:main} above, let $\varphi_m$ denote the homomorphism onto the finite action of $G$. The stabilizer of $[\omega_k] = [(0,\ldots,0,1,\ldots,1,0)] \in \Omega_{n,k,m}$ is $H_k\Ker(\varphi_m)$. The action of the parabolic subgroup $B\leq G$ is faithful on $\ZZ_3^n$ so the kernel is contained in the translation subgroup $\Ker(\varphi_m)\subseteq X$. In turn, by \eqref{eq:translation}, $\langle x_i^m\mid 1\leq i\leq n\rangle \leq \Ker(\varphi_m)$.
		
		If $m=1$ then it follows that $\Ker(\varphi_1)=X$ and the stabilizer is $L_k:=H_kX$. %=L_k$.
		
		If $m=2$ then the stabilizer $J_k=H_k\Ker(\varphi_2)$ is of index $2$ in $L_k$, the factor group is generated by $x_n$.

		We now prove that $J_k$ is a proto-Gelfand subgroup of $G$ which in turn implies that $L_k$ is a proto-Gelfand subgroup of $G$.

		The proof is along the lines of Theorem~\ref{thm:main2_new}.
		As $G=TVL_k=TV\langle x_n\rangle J_k$, every coset of $J_k$ in ${G}$ is of form ${\tau} {x_n}^h {w}{J_k}$ for some $\tau\in T$, $w\in V$ and $0\leq h\leq 1$. We can also assume that the expansion of $w$ in the natural basis of $V$ does not involve $e_j$ for $j\leq k-1$, as these are in $H_k\leq J_k$. Pick $w^\prime\in H_k\cap V$ such that $ww^\prime$ is $\tau$-invariant.
		
		If $\tau\ne id$ then there exists $k+1\leq i\leq n$ such that $i\ne\tau(i)\leq k$. Then $x_{\tau(i)}\in H_k$, $x_ix_n^{-1}\in J_k$ and $z=w^{-1}x_n^{-1}x_iw=(x_n^{-1})^wx_i^w\in J_k$ so $\tau x_n^h w z^h w^\prime (x_{\tau(i)}^{-h})\in \tau x_n^h w J_k$ is an involution:
		\begin{align*}(\tau  x_n^h wz^hw^\prime x_{\tau(i)}^{-h})^2=(\tau  x_i^h ww^\prime x_{\tau(i)}^{-h})^2&
		=x_{\tau(i)}^h ww^\prime x_i^{h}x_i^{-h}ww^\prime x_{\tau(i)}^{-h}=1.
		\end{align*}
		
		Finally, if $\tau=id$ then %$x_n^w=x_n^{w^\tau}=x_{\tau(n)}^{w^\tau}$ and 
		$\tau  x_n^h w J_k$ might not contain an involution, but its double coset $J_kx_n^hwJ_k$ %is still involutive
		contains an involution, since
		\begin{align*}(x_{n}^h w)^2=(x_{n}^w)^ hx_n^h\in\langle x_n^2\rangle\leq J_k,
		\end{align*}
		and thus
		\[
		id\in J_k\subseteq J_k (x_{n}^h w)^2 J_k \subseteq (J_k x_n^h w J_k)^2.
		\]

		\medskip
		
		%For the opposite direction, 
		If $m>2$ then the permutation character $\pi$ of $G$ on $\Omega_{n,k,m}$ is induced from the trivial character of the stabilizer $H_k\Ker(\varphi_m)$, $\pi=1_{H_k\Ker(\varphi_m)}^G$. Let $M=V_kL_k\leq N_G(H_k\Ker(\varphi_m))$ and $M/H_k\Ker(\varphi_m)$ is dihedral of order $2m\geq 6$. Then $1_{H_k\Ker(\varphi_m)}^M$ is not multiplicity-free, and hence nor is $\pi=(1_{H_k\Ker(\varphi_m)}^M)^G$. See Remark~\ref{rem:proto_HK}.
		
		As for the original combinatorial flip actions 
		on arc permutations, geometric caterpillars and triangle-free triangulations, 
		$m=n+2,n+3,n+4$ respectively, none of the actions is multiplicity-free.
		
	\end{proof}

	%\medskip
	
	%Analogously, for $m\leq 2$ the action of $\tB_n$ on each orbit $\Omega_{n,k,m}$ is multiplicity-free.\todo{PH: Possibly change it to a straightforward claim.}
	%Proposition~\ref{prop:m=2} is naturally generalized to type $\tB_n$.
	
	%\subsection{Proto-Gelfand subgroups of other Weyl groups} %{A unified approach}
	
	\subsection{Open problems}
	
	An early version of the current paper posed the problem of finding a unified construction of proto-Gelfand subgroups for all affine Weyl groups.
	An algebraic solution of this problem will be presented in~\cite{H}. 
	Combinatorial constructions of proto-Gelfand subgroups, %which are 
	essentially similar to those described in this paper,
	work to some extent for the other classical affine types, $\tA$ and $\tD$. 
	However, the results and proofs for these types are more
	complicated. % and will be described elsewhere.

	\medskip
	
	%A Coxeter system $(W,S)$ consists of a Coxeter group $W$ with a set $S$ of simple reflections.
	%Subgroups generated by subsets $J \subseteq S$ are called parabolic.
	The %proto-Gelfand subgroups 
	stabilizers described in this paper are reflection subgroups, i.e., generated by reflections (see Proposition~\ref{t:H_k}). 
	%but not by simple reflections; 
	%thus, they are not parabolic. It is well known that for finite types $A$ and $B$, 
	%all maximal parabolic subgroups are Gelfand; see, e.g., \cite[Prop. 45.3]{Bump}. 
	
	\begin{problem}
		Which pairs $(W,H)$, where $W$ is a Coxeter group and $H\le W$ a reflection subgroup or a double cover thereof, are proto-Gelfand ?
	\end{problem}
	
	%\todo{YR: Is this problem open?}
	
	%\bigskip
	
	%\subsection{Combinatorial applications} %{Combinatorial applications}
	
	%Combinatorial
	Geometric flip actions of $\tC_n$ on triangulations and trees served to motivate 
	the construction of proto-Gelfand subgroups of $\tC_n$.
	It is desired to find similar actions for other types.
	
	\begin{problem}
		Find natural %combinatorial 
		geometric flip actions of other affine Weyl groups.
	\end{problem}
	%It is desired to find similar combinatorial actions of affine Weyl groups of other types. %In particular, one may be interested in various combinatorial 
	%which realize
	Of special interest are %combinatorial 
	geometric realizations of Theorem~\ref{thm:main2_B}.
	%\ref{th:A_n}.\\
	
	%Applications to random walks ??

\end{document}